\documentclass[11pt]{amsart}

\usepackage{amscd,amssymb,amsopn,amsmath,amsthm,mathrsfs,graphics,amsfonts,pb-diagram,enumerate,verbatim,calc
}

\usepackage[all]{xy}

\usepackage{color}

\usepackage[OT2,OT1]{fontenc}
\newcommand\cyr{%
\renewcommand\rmdefault{wncyr}%
\renewcommand\sfdefault{wncyss}%
\renewcommand\encodingdefault{OT2}%
\normalfont
\selectfont}
\DeclareTextFontCommand{\textcyr}{\cyr}

\usepackage{amssymb,amsmath}

\DeclareFontFamily{OT1}{rsfs}{}
\DeclareFontShape{OT1}{rsfs}{n}{it}{<-> rsfs10}{}
\DeclareMathAlphabet{\mathscr}{OT1}{rsfs}{n}{it}

\numberwithin{equation}{section}
\newtheorem{theorem}{Theorem}[section]
\newtheorem{lemma}[theorem]{Lemma}
\newtheorem{proposition}[theorem]{Proposition}
\newtheorem{corollary}[theorem]{Corollary}

\newtheorem{question}{Question}

\newtheorem{definition and recall}[theorem]{Definition and recall}

\theoremstyle{definition}
\newtheorem{definition}[theorem]{Definition}
\newtheorem{remark}[theorem]{Remark}
\theoremstyle{remark}

\newtheorem{example}[theorem]{Example}

\newtheorem{notation}[theorem]{Notation}
\newtheorem{acknowledgement}{Acknowledgement}

\newcommand{\fp}{\frak{p}}

\begin{document}
\title[The structure of finitely generated modules]{On the structure of finitely generated modules and the unmixed degrees}

\author[N.T. Cuong]{Nguyen Tu Cuong}
\address{Institute of Mathematics, Vietnam Academy of Science and Technology, 18 Hoang Quoc Viet Road, 10307
Hanoi, Vietnam}
\email{ntcuong@math.ac.vn}

\author[P.H. Quy]{Pham Hung Quy${}^*$}
\address{Department of Mathematics FPT University, Hanoi, Vietnam}
\email{quyph@fe.edu.vn}
\thanks{$^*$Corresponding author: Pham Hung Quy
	\\
This work is partially supported by funds of Vietnam National Foundation for Science
and Technology Development (NAFOSTED) under grant number 101.04-2023.08}


\begin{abstract}
Let $(R, \frak m)$ be the homomorphic image of a Cohen-Macaulay local ring and $M$ a finitely generated $R$-module. We use the splitting of local cohomology to shed a new light on the structure of non-Cohen-Macaulay modules. Namely, we show that every finitely generated $R$-module $M$ is associated to a sequence of invariant modules. This module sequence expresses the deviation of $M$ with the Cohen-Macaulay property. Our result generalizes the unmixed theorem of Cohen-Macaulayness for any finitely generated $R$-module. As an application we construct a new extended degree in the sense of Vasconcelos.
\end{abstract}
\maketitle

\section{Introduction}
Throughout this paper, let $(R, \frak m)$ be a Noetherian local ring that is the homomorphic image of a Cohen-Macaulay local ring, and $M$ a finitely generated $R$-module of dimension $d$. Let $x_1, \ldots, x_d$ be a system of parameters of $M$.\\

The Cohen-Macaulay rings and modules are central objects of commutative algebra. The unmixed theorem says that $M$ is Cohen-Macaulay if and only if for every system of parameters $x_1, \ldots, x_d$ of $M$ for every $i < d$, we have that all associated prime ideals of $M/(x_1, \ldots, x_i)M$ have the same  dimension $d-i$, i.e. $\dim R/\frak p = d-i$ for all $\frak p \in \mathrm{Ass}\,M/(x_1, \ldots, x_i)M$. So $M/(x_1, \ldots, x_i)M$ is an unmixed module for all $i < d$. Note that if $\cap_{\frak p \in \mathrm{Ass}M}N(\frak p) = 0$ is a reduced primary decomposition of the zero submodule of $M$, then the {\it unmixed component} of $M$ is defined by
$$U_M(0) = \bigcap_{\frak p \in \mathrm{Ass}M, \dim R/\frak p = d}N(\frak p).$$
Then $U_M(0)$ is just the largest submodule of $M$ of dimension strictly less than $d$. The following is the unmixed component version of unmixed theorem.\\

\noindent {\bf The unmixed theorem.} {\it A finitely generated $R$-module $M$ is Cohen-Macaulay if and only if for some (and hence for all) system of parameters $x_1, \ldots, x_d$ of $M$ all unmixed components
$$U_M(0), U_{M/x_1M}(0), \ldots, U_{M/(x_1, \ldots, x_{d-1})M}(0)$$
are zero modules.}\\

A finitely generated $R$-module $M$ is Cohen-Macaulay if and only if every system of parameters $x_1, \ldots, x_d$ of $M$ is an $M$-regular sequence. Recall that $x_1, \ldots, x_d$ is an $M$-regular sequence if for all $i \le d$ all relations
$$x_1 a_1 + \cdots + x_i a_i = 0$$
 are trivial, i.e. $a_i \in (x_1, \ldots, x_{i-1})M$ for all $i \le d$. In general we have $a_i \in (x_1, \ldots, x_{i-1})M:x_i$, so $x_1, \ldots, x_d$ is an $M$-regular sequence if the sub-quotient  module
$$\frac{(x_1, \ldots, x_{i-1})M:x_i}{(x_1, \ldots, x_{i-1})M} = 0,$$
for all $i = 1, \ldots, d$. Since 
$$((x_1, \ldots, x_{i-1})M:x_i)/(x_1, \ldots, x_{i-1})M = 0:_{M/(x_1,\ldots ,x_{i-1})M}x_i$$ 
is a submodule of $M/(x_1, \ldots, x_{i-1})M$ of dimension less than or equal to $d-i = \dim M/(x_1, \ldots, x_{i-1})M -1$, we have
$$((x_1, \ldots, x_{i-1})M:x_i)/(x_1, \ldots, x_{i-1})M \subseteq U_{M/(x_1, \ldots, x_{i-1})M}(0)$$
for all $i \le d$. Set
$$\frak b(M) = \bigcap_{\underline{x}, i \le d} \mathrm{Ann} \frac{(x_1, \ldots, x_{i-1})M:x_i}{(x_1, \ldots, x_{i-1})M},$$
where $\underline{x} = x_1, \ldots, x_d$ runs over all systems of parameters of $M$.  It is clear that the ideal $\frak b(M)$ kills all non-trivial relations of systems of parameters of $M$.\\
 
The Cohen-Macaulayness of $M$ can be characterized by local cohomology: $M$ is Cohen-Macaulay if and only if the local cohomology $H^i_{\frak m}(M) = 0$ for all $i<d = \dim M$. Thus if $M$ is not Cohen-Macaulay, then $H^i_{\frak m}(M) \neq 0$ for some $i < d$. Notice that $H^i_{\frak m}(M)$ is always Artinian but it is rarely Noetherian. So $H^i_{\frak m}(M)$ may not be annihilated by $\frak m$-primary ideals. The ideals $\frak a_i(M) = \mathrm{Ann}\,H^i_{\frak m}(M)$, $i = 0 , \ldots, d$, play an important role in many areas in commutative algebra such as the homological conjectures, the tight closure theory, etc. Set $\frak a(M) = \frak
a_0(M) \ldots \frak a_{d-1}(M)$. Schenzel proved the following inclusions \cite[Satz 2.4.5]{Sch82}
$$\mathfrak{a}(M) \subseteq \mathfrak{b}(M) \subseteq \mathfrak{a}_0(M) \cap \cdots \cap \mathfrak{a}_{d-1}(M).$$
Notice that our ring is always the homomorphic image of a Cohen-Macaulay local ring. This condition gives us a critical fact that $\dim M/ \frak a(M)< \dim M$ for all finitely generated $R$-modules. Therefore, we can choose a parameter element $x$ contained in $\frak a(M)$ (and hence in $\frak b(M)$). Furthermore, we have a special system of parameters satisfying that
$$x_d \in \frak a(M), x_{d-1} \in \frak a(M/x_dM),  \ldots, x_1 \in \frak a(M/(x_2, \ldots, x_d)M).$$
Such a system of parameters is called a {\it $p$-standard system of parameters} \cite{C95}. The $p$-standard systems of parameters play a key ingredient in Kawasaki's proof for the Macaulayfication problem \cite{K00}. By \cite[Theorem 1.3 (a) and (b)]{CC17} $R$ is the homomorphic image of a Cohen-Macaulay local ring if and only if every finitely generated $R$-module admits a $p$-standard system of parameters.\\

 In this paper, we will use a similar type of $p$-standard system of parameters to study the splitting of local cohomology modules. As mentioned above we know that $0:x \subseteq U_M(0)$ for every parameter element $x$ of $M$. Moreover, if $x \in \frak b(M)$ then we have $0:x = U_M(0)$. Hence we get the following short exact sequence
$$0 \to M/U_M(0) \overset{x}{\to} M \to M/xM \to 0.$$
Furthermore if $x \in \frak b(M)^2$ then the above short exact sequence deduces the short exact sequence of local cohomology for any ideal $I$ (see Lemma \ref{B3.2.3})
$$0 \rightarrow H^i_I(M) \rightarrow H^i_I(M/xM) \rightarrow
H^{i+1}_I(M/U_M(0)) \rightarrow 0$$
for all $i < d - \dim R/I - 1$. Using the method of \cite{CQ11} we can study the splitting of these exact sequences of local cohomology modules. We obtain the first main result of this paper as follows.

\begin{theorem}\label{T1.1} Let $I$ be an ideal of $R$ and $x$ a parameter element of $M$ contained in $\frak b(M)^3$. Then for all $i < d - \dim R/I - 1$ we have
  $$H^i_I(M/xM) \cong H^i_I(M) \oplus H^{i+1}_I(M/U_M(0)).$$
\end{theorem}
This splitting result leads a new kind of system of parameters $x_1, \ldots, x_d$ satisfying that
$$x_d \in \frak b(M)^3, x_{d-1} \in \frak b(M/x_dM)^3,  \ldots, x_1 \in \frak b(M/(x_2, \ldots, x_d)M)^3.$$
We call such a system of parameters a {\it $C$-system of parameters} of $M$. Similar to $p$-standard system of parameters, every finitely generated $R$-module admits $C$-systems of parameters if and only if $R$ is a quotient of a Cohen-Macaulay local ring. It should be noted that the right hand sides of the above isomorphisms do not depend of the choice of $C$-parameter element $x \in \frak b(M)^3$. Thus the local cohomology modules $H^i_I(M/xM)$, $i < d - \dim R/I-1$, are invariants (up to an isomorphism). As consequences, we can expect several invariant properties of quotient modules $M/(x_i, \ldots, x_d)M$ regarding $C$-systems of parameters. For example, by using the fact $U_M(0) = H^0_{\frak b(M)}(M)$, we generalize the unmixed theorem for any finitely generated $R$-module to get the second main result of this paper. 

\begin{theorem}\label{T1.3}
Let $M$ be a finitely generated $R$-module of dimension $d$ and $\underline{x} = x_1, \ldots, x_d$ a $C$-system of parameters of $M$. Then the unmixed component $U_{M/(x_{i+1},
\ldots,x_d)M}(0)$ is independent of the choice of
$\underline{x}$ for all $1 \leq i \leq d$ (up to an isomorphism).
\end{theorem}

The above theorem assigns to any finitely generated $R$-module $M$ of dimension $d$ a sequence of modules $U_0(M), \ldots, U_{d-1}(M)$, which satisfies that
$U_i(M) \cong U_{M/(x_{i+2}, \ldots,x_d)M}(0)$ for every $C$-system of parameters $x_1, \ldots, x_d$ of $M$. Notice that $M$ is Cohen-Macaulay if and only if $U_i(M) = 0$ for all $i = 0, \ldots, d-1$ by the unmixed theorem. This module sequence gives information about the distance between $M$ and the Cohen-Macaulayness. We call $U_0(M), \ldots, U_{d-1}(M)$ the {\it Cohen-Macaulay deviated sequence} of $M$. The name of Cohen-Macaulay deviated sequence comes from the notion of {\it Cohen-Macaulay deviation} of Vasconcelos in his theory of extended degrees.\\

Let $I$ be an $\frak m$-primary ideal. We denote by $\mathrm{deg}(I, M)$ the ordinary multiplicity of $M$ with respect to $I$, and call it the {\it degree} of $M$ with respect to $I$. The degree, $\mathrm{deg}(I, M)$, is a basic invariant that measures the complexity of $M$ with respect to $I$. Vasconcelos et al. \cite{DGV98, V98-1, V98-2}
introduced the notion of {\it extended degree} in order to capture the size of a module along with some of the complexity of its
structure. It is a numerical function on the category of finitely generated modules
over local or graded rings which generalizes the ordinary degree. Let $\mathcal{M}(R)$ be the category of finitely generated
$R$-modules. An {\it extended degree} on $\mathcal{M}(R)$
with respect to $I$ is a numerical function
$$ \mathrm{Deg}(I, \bullet) : \mathcal{M}(R) \to \mathbb{R} $$
satisfying the following conditions
\begin{enumerate}[{(i)}]
\item $\mathrm{Deg} (I, M) = \mathrm{Deg}(I, \overline{ M}) + \ell(H^0_{\frak
m}(M))$, where $\overline{M} = M/H^0_{\frak m}(M)$.
\item (Bertini's rule) $\mathrm{Deg}(I, M) \geq \mathrm{Deg}(I, M/xM)$ for every generic element $x \in I\setminus \frak mI$ of $M$.
\item If $M$ is Cohen-Macaulay then $\mathrm{Deg}(I, M) =
\mathrm{deg}(I, M)$ the multiplicity of $M$ with respect to ideal $I$.
\end{enumerate}
The difference $\mathrm{Deg} (I, M) - \mathrm{deg} (I, M)$ is called the {\it Cohen-Macaulay deviation} of $M$ with respect to $I$. The prototype of an extended degree is the {\it homological degree}, $\mathrm{hdeg}(I, M)$, was introduced and studied by Vasconselos in  \cite{V98-1} (see Definition \ref{D3.3.4}). Using the Cohen-Macaulay deviated sequence we introduce the {\it unmixed degree} of $M$ with respect to
$I$, and denoted by $\mathrm{udeg}(I, M)$. We define
$$\mathrm{udeg}(I, M) = \mathrm{deg}(I, M) + \sum_{i=0}^{d-1}\delta_{i, \dim U_i(M)}\mathrm{deg}(I, U_i(M)),$$
where $\delta_{i, \dim U_i(M)}$ is Kronecker's symbol. The unmixed degree is a natural generalization of the ordinary degree as well as the {\it arithmetic degree} (for the definition of arithmetic degree, $\mathrm{adeg}(I,M)$, we refer to Definition \ref{adeg}). We obtain the last main result of this paper.

\begin{theorem}\label{T1.4}
  The unmixed degree $\mathrm{udeg}(I, \bullet)$ is an extended degree on the category of finitely generated $R$-modules $\mathcal{M}(R)$.
\end{theorem}

The paper is organized as follows. In the next section we collect useful results about the annihilator of local cohomology, the unmixed component and some special systems of parameters. We also mention the method of \cite{CQ11} to study the splitting of local cohomology. Section 3 is devoted to the splitting of local cohomology stated in Theorem \ref{T1.1} (see Theorem \ref{D3.2.4} and Corollary \ref{H3.2.5}). Then we introduce the notion of $C$-system of parameters, that plays a central role in this paper. Theorem \ref{T1.3} is proved in Section 4. We also prove the invariance of local cohomology of quotient modules regarding $C$-systems of parameters (cf. Theorem \ref{D3.2.7}). Some applications of the Cohen-Macaulay deviated sequence are given. The unmixed degree will be introduced in Section 5. Theorem \ref{T1.4} follows from Proposition \ref{M3.3.9}, Theorems \ref{D3.3.8} and \ref{D3.3.17}. The most difficult point is to prove the Bertini rule of unmixed degree. For that we show that for certain {\it superficial element} $x$ of $M$ with respect to $I$ we have $\mathrm{udeg}(I, M/xM) \le \mathrm{udeg}(I, M)$. We also compare the unmixed degree with the ordinary degree, the arithmetic degree and the homological degree.
\begin{acknowledgement}
	The authors are deeply grateful to the referee for carefully reading the manuscript and for her/his excellent comments. This paper was written during several visits of the second author at the Vietnam Institute for Advanced Study in Mathematics (VIASM). He would like to thank the VIASM for the very kind support and hospitality. 
\end{acknowledgement}
\section{Preliminaries}
 We start with the notion of the annihilator of local cohomology which will be used frequently in this paper.
\begin{notation} \rm Let $(R, \frak m)$ be a Noetherian local ring and $M$ a finitely generated $R$-module of dimension $d>0$.
\begin{enumerate}[{(i)}]
\item For all $i < d$ we set  $\frak a_i(M) =
\mathrm{Ann}H^{i}_\mathfrak{m}(M)$, and set $\frak a(M) = \frak
a_0(M) \ldots \frak a_{d-1}(M)$.
\item Put $\mathfrak{b}(M) = \bigcap_{\underline{x};i=1}^d
\mathrm{Ann}(0:x_i)_{M/(x_1,\ldots ,x_{i-1})M}$ where $\underline{x} =
x_1, \ldots, x_d$ runs over all systems of parameters of $M$.
\end{enumerate}
\end{notation}
\begin{remark}\label{C3.1.2}\rm
\begin{enumerate}[{(i)}]
\item Schenzel \cite[Satz 2.4.5]{Sch82} proved that
$$\mathfrak{a}(M) \subseteq \mathfrak{b}(M) \subseteq \mathfrak{a}_0(M) \cap \cdots \cap \mathfrak{a}_{d-1}(M).$$
\item If $R$ is the homomorphic image of a Cohen-Macaulay local ring, then $\dim R/\mathfrak{a}_i(M) \leq
i$ for all $i< d$ \cite[Theorem 1.3]{CC17}. Furthermore, $\dim R/\mathfrak{a}_i(M) = i$ if and only if there exists $\frak p
\in \mathrm{Ass}M$ such that $\dim R/\frak p = i$ (see \cite[Theorem
8.1.1]{BH98}).
\item If $R$ is the homomorphic image of a Cohen-Macaulay local ring, then Faltings' annihilator theorem claims that $\frak p \in
\mathrm{supp}(M)$ and $\frak p \notin V(\frak a(M))$ if and only if
$M_{\frak p}$ is Cohen-Macaulay and $\dim M_{\frak p} + \dim R/\frak
p = d$ (see \cite[9.6.6]{BS98}).
\item The condition that $R$ is the homomorphic image of a Cohen-Macaulay local ring can not be removed in  (ii) and (iii) by Nagata's example \cite[Example 2, pp. 203$-$205]{N62}.
\end{enumerate}
\end{remark}
Since we always assume that $(R, \frak m)$ is the homomorphic image of a Cohen-Macaulay local ring, Remark \ref{C3.1.2} (ii)
ensures $\dim R/\frak a(M) < d$. Therefore we can choose a parameter element $x \in \frak a(M)$. Following \cite{C95} such a parameter element is called {\it $p$-standard}.
\begin{definition}\rm  A system of parameters $x_1,\ldots,x_d$ of $M$ is called {\it $p$-standard} if $x_d \in \frak
a(M)$ and $x_i \in \frak a(M/(x_{i+1},\ldots,x_d)M)$ for all $i =
d-1,\ldots,1$.
\end{definition}
We recall a property of $p$-standard system of parameters which will be used in the sequel. Let
$\underline{x} = x_1,\ldots,x_d$ be a system of parameters of $M$. Let $\underline{n} = (n_1,\ldots,n_d)$ be a
$d$-tuple of positive integers and $\underline{x}^{\underline{n}} =
x_1^{n_1},\ldots,x_d^{n_d}$. We consider the difference
$$I_{M,\underline{x}}(\underline{n}) = \ell(M/(\underline{x}^{\underline{n}})M) -
e(\underline{x}^{\underline{n}};M)$$ as function in $\underline{n}$,
where $e(\underline{x};M)$ is the Serre multiplicity of $M$ with
respect to the sequence $\underline{x}$. Although
$I_{M,\underline{x}}(\underline{n})$ may be not a polynomial for
$n_1,\ldots,n_d$ large enough, it is bounded above by polynomials.
Moreover, the first author of this paper in \cite{C91} proved that the least degree of
all polynomials in $\underline{n}$ bounding above
$I_{M,\underline{x}}(\underline{n})$ is independent of the choice of
$\underline{x}$, and it is denoted by $p(M)$. The invariant $p(M)$
is called the {\it polynomial type} of $M$.
If $(R, \frak m)$ is the homomorphic image of a Cohen-Macaulay local ring, then $p(M) = \dim R/\frak a(M)$ when $M$ is equidimensional (see \cite{C92}). In addition, if
$\underline{x} = x_1,\ldots,x_d$ is $p$-standard then
we have the following.
\begin{proposition}[\cite{C95}, Theorem 2.6 (ii)]\label{M3.1.4}
Let $x_1,\ldots,x_d$ be a $p$-standard system of parameters of $M$. Then for all $n_1,\ldots,n_d>0$ we have
$$I_{M,\underline{x}}(\underline{n}) =
\sum_{i=0}^{p(M)} n_1\ldots n_i e_i,$$ where
$e_i = e(x_1,\ldots,x_i; 0:_{M/(x_{i+2},\ldots,x_d)M}x_{i+1})$ and $e_0 =
\ell(0:_{M/(x_{2},\ldots,x_d)M}x_{1})$.
\end{proposition}
In \cite{CC07-1}, D.T. Cuong and the first author introduced the notion of {\it $dd$-sequence} which is a special case of the notion of {\it $d$-sequences} of Huneke.
\begin{definition}[\cite{Hu82,GY86}]\rm
A sequence of elements $\underline{x} = x_1,\ldots,x_s$ is called a {\it
$d$-sequence} of $M$ if $(x_1,\ldots,x_{i-1})M:x_j =
(x_1,\ldots,x_{i-1})M:x_ix_j$ for all $i \leq j \leq s$.
A sequence $\underline{x} = x_1,\ldots,x_s$ is called a {\it
strong $d$-sequence} if $\underline{x}^{\underline{n}} =
x_1^{n_1},\ldots,x_s^{n_s}$ is a $d$-sequence for all $\underline{n} =
(n_1,\ldots,n_s) \in \mathbb{N}^s$.
\end{definition}

For important properties of $d$-sequence we refer to
\cite{Hu82,Tr83}.

\begin{definition}[\cite{CC07-1}]\rm
 A sequence of elements $\underline{x} = x_1,\ldots,x_s$ is call a {\it
$dd$-sequence} of $M$ if $\underline{x}$ is a strong $d$-sequence of
$M$ and the following conditions are satisfied
\begin{enumerate}[{(i)}]
\item $s=1$ or,
\item $s>1$ and $\underline{x}' = x_1,\ldots,x_{s-1}$ is a $dd$-sequence of $M/x_s^n$ for all $n \geq 1$.
\end{enumerate}
\end{definition}
The following is a characterization of $dd$-sequence in terms of $I_{M,\underline{x}}(\underline{n})$ (\cite[Theorem 1.2]{CC07-1}).
\begin{proposition}\label{M3.1.7}
A system of parameters $\underline{x} = x_1,\ldots,x_d$ of $M$ is a $dd$-sequence if and only if for all $n_1,\ldots,n_d>0$ we have
$$I_{M,\underline{x}}(\underline{n}) =
\sum_{i=0}^{p(M)} n_1\ldots n_i e_i,$$ where
$e_i = e(x_1,\ldots,x_i; 0:_{M/(x_{i+2},\ldots,x_d)M}x_{i+1})$ and $e_0 =
\ell(0:_{M/(x_{2},\ldots,x_d)M}x_{1})$.
\end{proposition}

\begin{remark}\rm \label{R dd seq}
\begin{enumerate}[{(i)}]
\item By Propositions \ref{M3.1.4} and \ref{M3.1.7}, if a system of parameter  $x_1,\ldots,x_d$ of $M$ is $p$-standard, then it is a $dd$-sequence. Conversely, if $x_1,\ldots,x_d$ is a $dd$-sequence then $x_1^{n_1},\ldots,x_d^{n_d}$ with $n_i \geq i, i = 1, \ldots,d$,
is $p$-standard (see \cite[Section
3]{CC07-1}).
\item  An $R$-module $M$ admits a $p$-standard (or $dd$-sequence) system of parameters if and only if $R/\mathrm{Ann}M$ is the homomorphic image of a
Cohen-Macaulay local ring  \cite[Theorem 1.3]{CC17}.
\end{enumerate}
\end{remark}
We next recall the notion of {\it unmixed component} of $M$ and its relations with the ideal $\frak b(M)$.
\begin{definition}\label{Dn3.2.1} \rm The largest submodule of $M$ of dimension less than $d$ is called the {\it unmixed component} of $M$, and denoted by $U_M(0)$.
\end{definition}

\begin{remark}\label{C3.2.2} \rm
\begin{enumerate}[{(i)}]
\item If $\cap_{\frak p \in \mathrm{Ass}M}N(\frak p) = 0$ is a reduced primary decomposition of the zero submodule of $M$, then
$U_M(0) = \cap_{\frak p \in \mathrm{Assh}M}N(\frak p)$, where  $\mathrm{Assh}M = \{\frak p \in \mathrm{Ass}M \mid \dim R/\frak p = d\}$.
\item Since $\dim U_M(0) < d$, there exists a parameter element $x$
of $M$ contained in $\mathrm{Ann}\, U_M(0)$. Therefore $U_M(0) \subseteq 0:x$.
But $x$ is a parameter element, so $\dim (0:x) < d$. Hence
$U_M(0) = 0:x$. Following the definition of $\frak b(M)$ we have $\frak b(M)
\subseteq \mathrm{Ann}U_M(0)$. Thus if $x \in \frak b(M)$ is a parameter element of $M$ then $U_M(0) = 0:x$. We also have $U_M(0) \cong H^0_{\frak b(M)}(M)$.
\item By (ii) we have $\cap_{x} \mathrm{Ann}(0:_Mx) =
\mathrm{Ann}U_M(0)$, where $x$ runs over all parameter elements of $M$. Therefore
\begin{eqnarray*}
\mathfrak{b}(M) &=& \bigcap_{\underline{x};i=1}^d
\mathrm{Ann}\,(0:x_i)_{M/(x_1,\ldots,x_{i-1})M}\\
&=& \bigcap_{\underline{x};i=1}^d
\mathrm{Ann}\,U_{M/(x_1,\ldots,x_{i-1})M}(0),
\end{eqnarray*}
where $\underline{x} = x_1,\ldots,x_d$ runs over all systems of parameters of $M$.
\end{enumerate}
\end{remark}

Problem of the splitting of local cohomology is started in \cite{CQ11}. For convenience we recall some results of \cite{CQ11} (with slight generalizations).
Suppose we are given an integer $t$, an ideal $\frak
a$ of $R$ and a submodule $U$ of $M$. Set $\overline{M} =M/U$. We
say that an element $x \in \mathfrak{a}$ satisfies the {\it (SES) condition} at degrees $i < t - 1$, the {\it short exact sequence condition} for local cohomology at degrees $i < t -1$, if $0:_Mx = U,$ and the short exact sequence
$$0 \longrightarrow \overline{M} \overset{x}{\longrightarrow} M \longrightarrow M/xM \longrightarrow 0$$
induces  short exact sequences
$$0 \longrightarrow H^{i}_\mathfrak{a}(M) \longrightarrow H^{i}_\mathfrak{a}(M/xM)
\longrightarrow H^{i+1}_\mathfrak{a}(\overline{M}) \longrightarrow
0$$ for all $i<t-1$. When this is the case, we consider the above
exact sequence as an extension of $H^{i+1}_\mathfrak{a}(\overline{M})$ by $H^{i}_\mathfrak{a}(M)$, therefore as an
 element of $\mathrm{Ext}^1_R(H^{i+1}_\mathfrak{a}(\overline{M}),
H^{i}_\mathfrak{a}(M))$ (see \cite[Chapter 3]{Mac75}). We denote this
element by $E_x^i$. Especially, if $H^{t}_\mathfrak{a}(\overline{M})
\cong H^{t}_\mathfrak{a}(M)$, then we have the short exact sequence
$$0 \longrightarrow H^{t-1}_\mathfrak{a}(M) \longrightarrow H^{t-1}_\mathfrak{a}(M/xM)
 \longrightarrow 0:_{H^{t}_\mathfrak{a}(\overline{M})}x \longrightarrow 0.$$
Let $\frak b$ be an ideal such that $x\in \frak b$. We denote by $F^{t-1}_x$ the element of
$\mathrm{Ext}^1_R(0:_{H^{t}_\mathfrak{a}(\overline{M})}\mathfrak{b},
0:_{H^{t-1}_\mathfrak{a}(M)}\mathfrak{b})$ which is represented by the following short exact sequence 
$$0 \longrightarrow 0:_{H^{t-1}_\mathfrak{a}(M)}\mathfrak{b} \longrightarrow
0:_{H^{t-1}_\mathfrak{a}(M/xM)}\mathfrak{b}
 \longrightarrow 0:_{H^{t}_\mathfrak{a}(\overline{M})}\mathfrak{b} \longrightarrow 0$$
provided that the
sequence is exact, and is determined by applying the $\mathrm{Hom}(R/\frak b, \bullet)$ functor.
 It should be noted that an extension of $R$-module $A$ by an $R$-module $C$ is split if it is the zero element of
 $\mathrm{Ext}^1_R(C, A)$. The two next theorems can be proven by the same method as used in \cite[Theorem 2.2]{CQ11}
\begin{theorem}\label{T2.13}
Let $t$ be a positive integer and $U$ a submodule of $M$. Let
$\overline{M} = M/U$. Suppose $x$ and $y$ are elements satisfying the (SES) condition at degrees $i < t - 1$, and $0:_M (x+y)=U$. Then
\begin{enumerate}[{(i)}]\rm
\item {\it $x+y$ also satisfies the (SES) condition at degrees $i < t - 1$, and $E_{x+y}^i =
E_x^i + E_y^i$ for all $i<t-1$.}
\item {\it If
$H^{t}_\mathfrak{a}(\overline{M}) \cong H^{t}_\mathfrak{a}(M)$ and
$F^{t-1}_x, F^{t-1}_{y}$ are determined, then $F^{t-1}_{x+y}$ is determined, and we have $F^{t-1}_{x+y} = F^{t-1}_x + F^{t-1}_{y}$.}
\end{enumerate}
 \end{theorem}

\begin{theorem}\label{T2.14} Let $t$ be a positive integer and $U$ a submodule of $M$. Let
$\overline{M} = M/U$. Suppose $x$ and $y$ are elements such that
$x$ satisfies the (SES) condition at degrees $i < t - 1$, and $0:_M xy=U$. Then
\begin{enumerate}[{\rm (i)}]
\item The element $xy$ satisfies the (SES) condition and $E_{xy}^i = yE_x^i$ for all $i<t-1$. Suppose that $H^{t}_\mathfrak{a}(\overline{M})
\cong H^{t}_\mathfrak{a}(M)$ and $F^{t-1}_x$ is determined.
Then $F^{t-1}_{xy}$ is determined and $F^{t-1}_{xy}=yF^{t-1}_x$.
\item {\it Suppose that $H^{t}_\mathfrak{a}(\overline{M}) \cong
H^{t}_\mathfrak{a}(M)$ and $yH^{i}_\mathfrak{a}(M)=0$ for all $i<t$.
Then $E_{xy}^i =0$ for all $i<t-1$. Moreover, $F^{t-1}_{xy}$ is determined and $F^{t-1}_{xy} = 0$.}
\end{enumerate}
\end{theorem}

The following is a prime avoidance theorem for a product of ideals.
\begin{lemma}[\cite{CQ11} Lemma 3.1] \label{L2.15}
Let $(R, \mathfrak{m})$ be a Noetherian local ring, $\mathfrak{a}$,
$\mathfrak{b}$  ideals  and $\mathfrak{p}_1, \ldots, \mathfrak{p}_n$
prime ideals such that $\mathfrak{ab} \nsubseteq \mathfrak{p}_j$ for
all $j \leq n$. Let $x \in \mathfrak{ab}$ with $x \notin
\mathfrak{p}_j$ for all $j \leq n$. Then, there are elements  $a_1,
\ldots, a_r \in \mathfrak{a}$ and $ b_1, \ldots, b_r \in \mathfrak{b}$ such that
$x=a_1b_1+ \cdots + a_rb_r$, and that $a_ib_i \notin \mathfrak{p}_j$
and $a_1b_1+ \cdots +a_ib_i \notin \mathfrak{p}_j$ for all $i \leq
r$ and all $j \leq n$.
\end{lemma}

\begin{corollary}\label{C2.16}
  Let $(R, \mathfrak{m})$ be a Noetherian local ring, $M$ a finitely generated $R$-module of dimension $d>0$, $\frak a$ and $\frak b$ two ideals such that $\dim R/\frak a \frak b<d$.
  Let $x \in \frak a \frak b$ be a parameter element of $M$. Then, there exist parameter elements $a_1,
\ldots, a_r \in \frak a$ and $b_1, \ldots, b_r \in \mathfrak{b}$ of $M$ such that
$x=a_1b_1+ \cdots + a_rb_r$, and that $a_1b_1+ \cdots +a_ib_i $ is a parameter element for all $i \leq
r$.
\end{corollary}
\begin{proof} Note that an element $x$ is a parameter element of $M$ if and only if $x \notin \frak p$ for all $\fp \in \mathrm{Assh}\,M$. The assertion now follows from Lemma \ref{L2.15}.
\end{proof}
\section{The splitting of local cohomology}
In this section we prove splitting theorems for local cohomology in local rings. These results lead to a new kind of systems of parameters. We need the following key ingredient about the annihilator of local cohomology supported at an arbitrary ideal that is of independent interest.
\begin{proposition}\label{M3.1.11}
Let $M$ be a finitely generated $R$-module of dimension $d$ and $I \supseteq \mathrm{Ann}\, M$ an ideal of $R$. We have $\mathfrak{b}(M) H^i_{I}(M)=0$
for all $i < d - \dim R/I$.
\end{proposition}
To prove the above result we use the following isomorphism of Nagel and
Schenzel in \cite[Proposition 3.4]{NS94}, see also \cite[Theorem 2.7]{HQ19}. Recall that a sequence $x_1,\ldots,x_t$ of elements contained in $I$ is an
{\it $I$-filter regular sequence} of $M$ if
$$\mathrm{Supp}\,
((x_1,\ldots,x_{i-1})M:x_i)/(x_1,\ldots,x_{i-1})M \subseteq V(I)$$
for all $i = 1,\ldots,t$, where $V(I)$ denotes the set of prime
ideals containing $I$. This condition is equivalent to that $x_i \notin \fp$ for all $\fp \in \mathrm{Ass}\, M/(x_1, \ldots, x_{i-1})M \setminus V(I)$, and for all $i = 1, \ldots, t$. Moreover we can choose an $I$-filter regular sequence on $M$ of any length by the prime avoidance theorem.
\begin{lemma}[Nagel-Schenzel's isomorphism]\label{B3.1.12}
Let $I$ be an ideal of $R$ and $x_1, \ldots, x_t$ an $I$-filter regular
sequence of $M$. Then we have
$$
H^j_{I}(M)\cong
\begin{cases}
H^j_{(x_1,\ldots,x_t)}(M) \quad \quad \quad\,\, \text{with}\,\, j<t\\
H^{j-t}_I(H^t_{(x_1,\ldots,x_t)}(M))\,\, \text{with}\,\, j\geq t.\\
\end{cases}
$$
\end{lemma}

\begin{proof}[Proof of Proposition \ref{M3.1.11}] Set $ t = d - \dim R/I$. Suppose $t > 0$,  by the prime avoidance theorem we can choose an element $x_1 \in I$ such that $x_1 \notin \frak p$ for all $\frak p \in \mathrm{Assh}\, M \cup (\mathrm{Ass}\, M \setminus V(I) )$. Thus $x_1$ is a parameter element of $M$ that is also an $I$-filter regular element. We continue this process to obtain a part of a system of parameters $x_1, \ldots, x_t$ of $M$ that is also an $I$-filter regular sequence on $M$. By Lemma \ref{B3.1.12}, for
$i<t$, we have
\begin{eqnarray*}
H^i_{I}(M)   &\cong& H^{0}_I(H^i_{(x_1,\ldots,x_i)}(M))\\
        &\cong& H^{0}_I(\lim_{\longrightarrow}M/{(x_1^n,\ldots,x_i^n)}M)\\
        &\cong& \lim_{\longrightarrow} H^{0}_I(M/{(x_1^n,\ldots,x_i^n)}M)\\
        &\cong& \lim_{\longrightarrow}
        \frac{(x_1^n,\ldots,x_i^n)M : I^\infty}{(x_1^n,\ldots,x_i^n)M}\\
        &\cong&  \lim_{\longrightarrow}
        \frac{(x_1^n,\ldots,x_i^n)M :
        x_{i+1}^\infty}{(x_1^n,\ldots,x_i^n)M},
\end{eqnarray*}
where $(x_1^n,\ldots,x_i^n)M : I^\infty = \cup_{k \ge 1} (x_1^n,\ldots,x_i^n)M : I^k$. Since $x_1, \ldots, x_t$ is a part of a system of parameters of $M$ and $(x_1^n,\ldots,x_i^n)M : x_{i+1}^\infty = (x_1^n,\ldots,x_i^n)M :
        x_{i+1}^k$ for some $k$, we have
$$\frak b(M) \frac{(x_1^n,\ldots,x_i^n)M :
        x_{i+1}^\infty}{(x_1^n,\ldots,x_i^n)M} = 0$$
for all $n \ge 1$ by the definition of $\frak b(M)$. Hence the equality $\mathfrak{b}(M) H^i_{I}(M)=0$ holds for all
$i < d - \dim R/I$. The proof is complete.
\end{proof}

\begin{lemma}\label{B3.2.3}
Let $I$ be an ideal of $R$ and $x, y \in \mathfrak{b}(M)$ parameter elements of $M$. Let $U_M(0)$ be the unmixed component of $M$. Put $\overline{M} = M/U_M(0)$ and $t = d
-\dim R/I$. Then for all $i<t-1$ we have the following short exact sequence
$$0 \rightarrow H^i_I(M) \rightarrow H^i_I(M/xyM) \rightarrow
H^{i+1}_I(\overline{M}) \rightarrow 0.$$ Furthermore, if $H^{t}_I(M)
\cong H^{t}_I(\overline{M})$ then we have the short exact sequence
$$0 \rightarrow H^{t-1}_I(M) \rightarrow H^{t-1}_I(M/xyM) \rightarrow
0:_{H^{t}_I(M)}xy \rightarrow 0.$$
\end{lemma}

\begin{proof}
By Remark \ref{C3.2.2} (ii) we have $U_M(0) = 0:_Mx = 0:_Mxy$. Therefore the following diagram commutes
\[\divide\dgARROWLENGTH by 2
\begin{diagram}
\node{0}\arrow{e}\node{\overline{M}}
\arrow{e,t}{x}\arrow{s,l}{\mathrm{id}}\node{M}\arrow{e}\arrow{s,l}{y}
\node{M/xM}\arrow{s}\arrow{e} \node{0}\\
\node{0}\arrow{e}\node{\overline{M}}
\arrow{e,t}{xy}\node{M}\arrow{e}\node{M/xyM}\arrow{e} \node{0.}
\end{diagram}
\]
Applying the functor $H^{i}_I(\bullet)$ to the above diagram we obtain the following commutative diagram for all $i < t-1$
\[\divide\dgARROWLENGTH by 2
\begin{diagram}
\node{\cdots}\arrow{e}\node{H^{i}_I(\overline{M})}
\arrow{e,t}{\psi^i}\arrow{s,l}{\mathrm{id}}\node{H^{i}_I(M)}\arrow{e}\arrow{s,l}
{y}
\node{H^{i}_I(M/xM)}\arrow{s}\arrow{e} \node{\cdots}\\
\node{\cdots}\arrow{e}\node{H^{i}_I(\overline{M})}
\arrow{e,t}{\varphi^i}\node{H^{i}_I(M)}\arrow{e}\node{H^{i}_I(M/xyM)}\arrow{e}
\node{\cdots,}
\end{diagram}
\]
where $\psi^i$ and $\varphi^i$ are derived from homomorphisms $\overline{M }\overset{x}{\to}M$ and $\overline{M}
\overset{xy}{\to}M$, respectively. By Proposition \ref{M3.1.11},
$yH^{i}_I(M)=0$ for all $i \leq t-1$, which implies $\varphi^i =0$ for all $i
\leq t-1$. Thus we have the short exact sequences
$$0 \rightarrow H^i_I(M) \rightarrow H^i_I(M/xyM) \rightarrow
H^{i+1}_I(\overline{M}) \rightarrow 0$$ for all $i < t-1$. For $i = t-1$ we have the exact sequence
$$0 \rightarrow H^{t-1}_I(M) \rightarrow H^{t-1}_I(M/xyM) \rightarrow
H^{t}_I(\overline{M}) \overset{xy}{\rightarrow} H^t_I(M).$$
Moreover,
if $H^{t}_I(M) \cong H^{t}_I(\overline{M})$ then we get the following short exact sequence
$$0 \rightarrow H^{t-1}_I(M) \rightarrow H^{t-1}_I(M/xyM) \rightarrow
0:_{H^{t}_I(M)}xy \rightarrow 0.$$
\end{proof}

Let $xy$ be a parameter element of $M$ such that $x, y \in \frak b(M)$. Lemma
\ref{B3.2.3} says that
$xy$ satisfies the (SES) condition at degrees $i < t-1$ with $t = d - \dim R/I$ and $U = U_M(0)$ as mentioned in Section 2.
Let $x \in \mathfrak{b}(M)^2$ be a parameter element of $M$, for all $i
< t-1$, we denote by $E^i_x$ the element in
$\mathrm{Ext}^1_R(H^{i+1}_I(\overline{M}), H^i_I(M))$ represented by the following short exact sequence provided it is determined
$$0 \rightarrow H^i_I(M) \rightarrow H^i_I(M/xM) \rightarrow
H^{i+1}_I(\overline{M}) \rightarrow 0.$$ In the case $i=t-1$ and
assume that $H^{t}_I(M) \cong H^{t}_I(\overline{M})$, we have the short exact sequence
$$0 \rightarrow H^{t-1}_I(M) \rightarrow H^{t-1}_I(M/xM) \rightarrow
0:_{H^{t}_I(M)}x \rightarrow 0.$$
Suppose we obtain the following short exact sequence by applying the $\mathrm{Hom}(R/ \frak b(M), \bullet)$ to above short exact sequence
$$0 \rightarrow H^{t-1}_I(M) \rightarrow 0:_{H^{t-1}_I(M/xM)}\mathfrak{b}(M)
 \rightarrow
0:_{H^{t}_I(M)}\mathfrak{b}(M) \rightarrow 0.$$
Then we denote by $F^{t-1}_{x}$ the element of $\mathrm{Ext}^1_R(0:_{H^{t}_I(M)}\mathfrak{b}(M), H^{t-1}_I(M))$
represented by the above short exact sequence. The main result of this section is as follows.
\begin{theorem}\label{D3.2.4} Let $M$ be a finitely generated $R$-module of dimension $d$, $I \supseteq \mathrm{Ann}\, M$ an ideal of $R$ and
$x$ a parameter element of $M$. Let $U_M(0)$ be the unmixed component of $M$ and set $\overline{M} = M/U_M(0)$. Let
$t = d -\dim R/I$. Then
\begin{enumerate}[{(i)}]\rm
\item {\it If $x \in \mathfrak{b}(M)^2$ then $E^i_x$ is determined for all $i<t-1$.}
\item {\it If $x \in \mathfrak{b}(M)^3$ then $E^i_x = 0$ for all $i<t-1$. Moreover, if $H^{t}_I(M) \cong H^{t}_I(\overline{M})$
then $F^{t-1}_{x} = 0$.}
\end{enumerate}
\end{theorem}
\begin{proof} (i) Notice that $\frak b(M) \nsubseteq \frak p$ for all $\frak p \in \mathrm{Assh}\,M$. By Corollary \ref{C2.16} there exist parameter elements $a_1,
\ldots, a_r, b_1, \ldots, b_r \in \mathfrak{b}(M)$ of $M$ such that
$x=a_1b_1+ \cdots + a_rb_r$, and $a_1b_1+ \cdots +a_jb_j $ are parameter elements for all $j \leq
r$. By Lemma \ref{B3.2.3} $E^i_{a_kb_k}$ is determined for all $i < t-1$ and for all $1 \leq k \leq r$. By Theorem \ref{T2.13} we have that
$$E^i_x = E^i_{a_1b_1} + \cdots +
E^i_{a_rb_r}$$
is determined for all $i < t-1$.\\
(ii) Similarly, we choose parameter elements $a_1,
\ldots, a_r \in \frak b(M)^2$ and $b_1, \ldots, b_r \in \mathfrak{b}(M)$ of $M$ such that
$x=a_1b_1+ \cdots + a_rb_r$, and $a_1b_1+ \cdots +a_jb_j $ are parameter elements for all $j \leq
r$. By Theorem \ref{T2.14} (ii) we have $E^i_{a_kb_k} = 0$ for all $i < t-1$ and for all $1 \leq k \leq r$. So $E^i_x = 0$ for all $i<t-1$.\\
 For the last assertion, by the same method, it is sufficient to show that $F^{t-1}_{ab} = 0$ for all parameter elements $a \in \frak b(M)^2$ and $b \in \frak b(M)$ provided $H^t_I(M) \cong H^t_I(\overline{M})$. Indeed, since $E^i_a$ and $E^i_{ab}$ are determined for all $i<t-1$, the commutative diagram
\[\divide\dgARROWLENGTH by 2
\begin{diagram}
\node{0}\arrow{e}\node{\overline{M}}
\arrow{e,t}{a}\arrow{s,l}{\mathrm{id}}\node{M}\arrow{e}\arrow{s,l}{b}
\node{M/aM}\arrow{s}\arrow{e} \node{0}\\
\node{0}\arrow{e}\node{\overline{M}}
\arrow{e,t}{ab}\node{M}\arrow{e}\node{M/abM}\arrow{e} \node{0,}
\end{diagram}
\]
yields the following diagram
\[\divide\dgARROWLENGTH by 2
\begin{diagram}
\node{0}\arrow{e}\node{H^{t-1}_I(M)}
\arrow{e,t}{i}\arrow{s,l}{b}\node{ H^{t-1}_I(M/aM)}\arrow{e}\arrow{s,l}{\beta}
\node{0:_{H^{t}_I(M)}a}\arrow{s,l}{\alpha}\arrow{e} \node{0}\\
\node{0}\arrow{e}\node{H^{t-1}_I(M)}
\arrow{e,t}{\delta}\node{ H^{t-1}_I(M/abM)}\arrow{e,t}{\pi}\node{0:_{H^{t}_I(M)}ab}\arrow{e} \node{0,}
\end{diagram}
\]
where $\alpha: 0:_{H^{t}_I(M)}a \to 0:_{H^{t}_I(M)}ab$ is injective. By Proposition \ref{M3.1.11} $b H^{t-1}_I(M) = 0$, so $\beta \circ i = 0$. Thus we have a homomorphism $\epsilon: 0:_{H^{t}_I(M)}a \to  H^{t-1}_I(M/abM)$ which makes the following diagram commute
\[\divide\dgARROWLENGTH by 2
\begin{diagram}
\node{0}\arrow{e}\node{H^{t-1}_I(M)}
\arrow{e,t}{i}\arrow{s,l}{b}\node{ H^{t-1}_I(M/aM)}\arrow{e}\arrow{s,l}{\beta}
\node{0:_{H^{t}_I(M)}a}\arrow{s,l}{\alpha}\arrow{sw,t}{\epsilon}\arrow{e} \node{0}\\
\node{0}\arrow{e}\node{H^{t-1}_I(M)}
\arrow{e,t}{\delta}\node{ H^{t-1}_I(M/abM)}\arrow{e,t}{\pi}\node{0:_{H^{t}_I(M)}ab}\arrow{e} \node{0,}
\end{diagram}
\]
By applying the $\mathrm{Hom}_R(R/\frak b(M), \bullet)$ to the above diagram we have the following diagram
\[\divide\dgARROWLENGTH by 2
\begin{diagram}
\node{}\node{} \node{}
\node{0:_{H^{t}_I(M)} \frak b(M)}\arrow{sw,t}{\epsilon}\arrow{s,l}{\mathrm{id}}\\
\node{0}\arrow{e}\node{H^{t-1}_I(M)}
\arrow{e}\node{ 0:_{H^{t-1}_I(M/abM)} \frak b(M)}\arrow{e,t}{\pi}\node{0:_{H^{t}_I(M)}\frak b(M),}
\end{diagram}
\]
where the row is an exact sequence, and the vertical map is  the identity
map. Since $\pi \circ \epsilon = \mathrm{id}$, the homomorphism $\pi$ is split. Thus $F^{t-1}_{ab} = 0$. The proof is complete.
\end{proof}
In the case $I = \frak  m$, we obtain a generalization of \cite[Corollary 4.1]{CQ11} and \cite[Proposition 3.4]{Q12}.
\begin{corollary}\label{H3.2.5} Let $x \in \mathfrak{b}(M)^3$ be a parameter element of $M$. Let $U_M(0)$ be the unmixed component of $M$ and set $\overline{M} = M/U_M(0)$.
 Then
$$H^i_{\mathfrak{m}}(M/xM) \cong H^i_{\mathfrak{m}}(M) \oplus H^{i+1}_{\mathfrak{m}}(\overline{M})$$
for all $i<d-1$, and
$$0:_{H^{d-1}_{\mathfrak{m}}(M/xM)}\mathfrak{b}(M)  \cong H^{d-1}_{\mathfrak{m}}(M) \oplus
0:_{H^{d}_{\mathfrak{m}}(M)}\mathfrak{b}(M).$$
\end{corollary}

See \cite{CQ20} for an application of the Splitting Theorems on the invariant of index of reducibility of parameter ideals. By the above Splitting Theorems, it is natural to consider the following system of parameters.
\begin{definition}[\cite{MQ16}, Definition 2.15] \rm A parameter element $x\in \frak b(M)^3$ is called a {\it $C$-parameter element} of $M$. A system of parameters $x_1, ..., x_d$ is called a {\it $C$-system of parameters} of $M$ if $x_d \in \mathfrak b(M)^3$ and $x_i \in \mathfrak b(M/(x_{i+1}, ..., x_d)M)^3$ for all $i = d-1, ..., 1$. A sequence of elements $x_i, \ldots, x_d$ is called {\it a part of $C$-system of parameters} if we can expand it to a $C$-system of parameters $x_1, \ldots, x_d$.
\end{definition}

It is evident that $C$-systems of parameters are closely related with $p$-standard systems of parameters. Lemmas below will be very useful in the sequel.

\begin{lemma} \label{B3.1.9}
Let $x$ be a parameter element of $M$. Then $\frak
b(M) \subseteq \frak b(M/xM)$.
\end{lemma}
\begin{proof}
It follows from the definition of $\frak b(M)$.
\end{proof}
\begin{lemma}\label{B3.1.10} Let $x_1, \ldots, x_d$ be a $C$-system of parameters of $M$. Then $x_1, \ldots, x_{j-1},x_{j+1},\ldots,x_d$
is a $C$-system of parameters of $M/x_jM$ for all $j \leq d$.
\end{lemma}
\begin{proof}
The case $j=d$ is clear. For $j \neq d$ by Lemma
\ref{B3.1.9} the inclusion $\frak b(M) \subseteq \frak b(M/x_jM)$ holds. Therefore $x_d$ is a $C$-parameter element of $M/x_jM$. Notice that
$x_1, \ldots,x_{d-1}$ is a $C$-system of parameters of $M/x_dM$. The claim follows from the induction on $d$.
\end{proof}

\section{The Cohen-Macaulay deviated sequences}

In this section, we use the Splitting Theorem \ref{D3.2.4} to shed a new light on the structure of non-Cohen-Macaulay modules. Let $M$ be a finitely generated $R$-module of dimension $d$. The unmixed characterization of Cohen-Macaulay modules says that $M$ is Cohen-Macaulay if and only if for some (and hence for all) system of parameters
$x_1,\ldots,x_d$ we have $U_{M/(x_{i+1},\ldots,x_d)M}(0) = 0$ for all $1 \leq i
\leq d$. If $M$ is a generalized Cohen-Macaulay module, and $n_0$ is an integer such
that $\frak
m^{n_0}H^i_{\mathfrak{m}}(M) = 0$ for all $i<d$, then by \cite[Corollary 4.2]{CQ11}
we have
$$U_{M/(x_{i+1},\ldots,x_d)M}(0) = H^0_{\mathfrak{m}}(M/(x_{i+1},\ldots,x_d)M)
\cong \bigoplus_{j=0}^{d-i} H^j_{\mathfrak{m}}(M)^{\binom{d-i}{j}},$$
for any system of parameters $x_1,\ldots,x_d \in \frak m^{2n_0}$. Thus $U_{M/(x_{i+1},\ldots,x_d)M}(0)$ is independent of the choice of system of parameters $x_1,\ldots,x_d$ contained in
$\frak m^{2n_0}$ for all $1 \leq i \leq d$ (up to an isomorphism). The aim of this section is to generalize this fact for any finitely generated $R$-module. Concretely, we will show that for all $1 \leq i \leq d$ the modules $U_{M/(x_{i+1},\ldots,x_d)M}(0)$
is independent (up to an isomorphism) of the choice of a $C$-system of parameters $x_1,
\ldots, x_d$. We start with the following result about the invariance of local cohomology of quotient modules regarding $C$-systems of parameters.

\begin{theorem} \label{D3.2.7}
Let $\underline{x} = x_1, \ldots, x_d$ be a $C$-system of parameters of $M$. Then the local cohomology module $H^j_{\mathfrak{m}}(M/(x_{i+1}, \ldots,x_d)M)$
is independent of the choice of $\underline{x}$ for all $j
< i < d$ (up to an isomorphism).
\end{theorem}
\begin{proof}
We set $M_i = M/(x_{i+1}, \ldots,x_d)M$ for all $i < d$. We consider another $C$-system of parameters $\underline{y} = y_1, \ldots, y_d$ of $M$, and put $M_i'
= M/(y_{i+1},\ldots,y_d)M$ for all $i < d$. We proceed by induction on $d$ that $H^j_{\mathfrak{m}}(M_i) \cong
H^j_{\mathfrak{m}}(M_i')$
for all $j < i < d$. The assertion is trivial if $d=1$. For $d>1$ and $i=d-1$ since $x_d$ and $y_d$ are $C$-parameter elements, Corollary \ref{H3.2.5} implies that
$$H^j_{\mathfrak{m}}(M_{d-1}) \cong H^j_{\mathfrak{m}}(M) \oplus
H^{j+1}_{\mathfrak{m}}(M/U_M(0)) \cong
H^j_{\mathfrak{m}}(M_{d-1}')$$ for all $j<d-1$. Suppose $i< d-1$. Since
$\dim R/\frak b(M_{i+1}) < \dim M_{i+1} = i+1$ and $\dim R/\frak b(M_{i+1}') < \dim M'_{i+1} = i+1$
we can choose a $C$-parameter element
$z$ of both $M_{i+1}$ and $M_{i+1}'$. By the inductive hypothesis we have
\[ H^j_{\mathfrak{m}}(M_{i}) = H^j_{\mathfrak{m}}(M_{i+1}/x_{i+1}M_{i+1})
\cong H^j_{\mathfrak{m}}(M/(z, x_{i+2},\ldots,x_d)M) \quad \quad \quad \quad \quad \quad \quad (1) \]
and
$$H^j_{\mathfrak{m}}(M_{i}') = H^j_{\mathfrak{m}}(M_{i+1}'/y_{i+1}M_{i+1}') \cong
H^j_{\mathfrak{m}}(M/(z, y_{i+2},\ldots,y_d)M)\quad \quad \quad \quad \quad \quad \quad (2)$$ for all $j<i$. Notice that $z, x_{i+2},\ldots,x_d$ and
$z, y_{i+2},\ldots,y_d$ are parts of $C$-systems of parameters of $M$. By Lemma \ref{B3.1.10} we have $x_{i+2},\ldots,x_d$ and
$y_{i+2},\ldots,y_d$ are parts of $C$-systems of parameters of $M/zM$. Applying the inductive hypothesis for $M/zM$ we have
$$H^j_{\mathfrak{m}}(M/(z, x_{i+2},\ldots,x_d)M) \cong H^j_{\mathfrak{m}}(M/(z, y_{i+2},\ldots,y_d)M) \quad \quad \quad \quad \quad \quad \quad \quad \quad \quad (3)$$
for all $j<i$. The assertion follows from the isomorphisms $(1)$, $(2)$ and $(3)$.
\end{proof}

\begin{corollary}\label{C invar ann}
  Let $\underline{x} = x_1, \ldots, x_d$ be a $C$-system of parameters of $M$. Then for all $i < d$, the ideals $\frak a(M/(x_{i+1}, \ldots,x_d)M)$ and $\sqrt{\frak a(M/(x_{i+1}, \ldots,x_d)M)} = \sqrt{\frak b(M/(x_{i+1}, \ldots,x_d)M)}$ are independent of the choice of $\underline{x}$.
\end{corollary}

We need the following result.
\begin{lemma}\label{B3.2.8}
Let $x $ be a $C$-parameter element of $M$. Then
$U_{M/xM}(0)$ is independent of the choice of $x$ (up to an isomorphism).
\end{lemma}
\begin{proof}
By Corollary \ref{C invar ann}, we have the ideal
$$\mathfrak{b}' = \sqrt{\mathfrak{a}(M/xM)} = \sqrt{\mathfrak{b}(M/xM)}$$
is independent of the choice of $C$-parameter element $x$. We have $U_{M/xM}(0) \cong
H^0_{\mathfrak{b}'}(M/xM)$ by Remark
\ref{C3.2.2} (ii). Since $\dim R/\mathfrak{b}' \leq \dim M/xM
-1 = d-2$, Theorem \ref{D3.2.4} (ii) implies that
$$H^0_{\mathfrak{b}'}(M/xM) \cong H^0_{\mathfrak{b}'}(M) \oplus H^1_{\mathfrak{b}'}(M/U_M(0)),$$
and the right hand side does not depend on $x$. Thus the unmixed component $U_{M/xM}(0)$ is independent of the choice of $C$-parameter element $x$
(up to an isomorphism).
\end{proof}

Using Lemma \ref{B3.2.8} and by the same method as used in the proof of Theorem
\ref{D3.2.7} we obtain the main result of this section.
\begin{theorem}\label{D3.2.9}
Let $M$ be a finitely generated $R$-module of dimension $d$ and $\underline{x} = x_1, \ldots, x_d$ a $C$-system of parameters of $M$. Then for all $1 \leq i \leq d$, the unmixed component $U_{M/(x_{i+1},
\ldots,x_d)M}(0)$ is independent of the choice of
$\underline{x}$ (up to an isomorphism).
\end{theorem}
\begin{definition} \rm
For all $0 \leq i \leq d-1$ we denote by $U_i(M)$ the module satisfying that $U_i(M) \cong U_{M/(x_{i+2}, \ldots,x_d)M}(0)$
for all $C$-systems of parameters $x_1, \ldots, x_d$ of $M$.
Notice that $\dim U_i(M) \leq i$ for all $0 \leq i \leq d-1$, and $U_{d-1}(M) \cong U_M(0)$. We call the module sequence $U_0(M), \ldots, U_{d-1}(M)$ the {\it Cohen-Macaulay deviated sequence} of $M$. 
\end{definition}
Notice that the Cohen-Macaulay deviated sequence of $M$ is the zero sequence if and only if $M$ is Cohen-Macaulay. We next use the Cohen-Macaulay deviated sequence to prove some properties of $C$-systems of parameters.
\begin{corollary}\label{H3.2.11}
Let $\underline{x} = x_i, \ldots, x_d, i > 1$, be a part of $C$-system of parameters of $M$. Then $\mathfrak{b}(M/(x_i, \ldots,x_d)M) =
\mathfrak{b}(M/(x_i^{n_i}, \ldots,x_d^{n_d})M)$ where $n_j \geq 1$
for all $i \leq j  \leq d$.
\end{corollary}
\begin{proof}
For $i = d$, notice that $\underline{y} = y_1,\ldots,y_{d-1}$ is a system of parameters of $M/x_dM$ if and only if it is also a system of parameters of
$M/x_d^{n_d}M$ for all $n_d \geq 1$. By Lemma \ref{B3.1.9} we have $x_d$
and hence $x_d^{n_d}$ are contained in $ \frak b(M/(y_1,\ldots,y_{j-1})M)^3$ for all $1 \leq j \leq d-1$. So Theorem \ref{D3.2.9} claims that
$$U_{M/(y_1,\ldots,y_{j-1},x_d)M}(0)  \cong
U_{M/(y_1,\ldots,y_{j-1},x_d^{n_d})M}(0)$$
for all $1 \leq j \leq d-1$.
By Remark \ref{C3.2.2} (iii) we have
\begin{eqnarray*}
\mathfrak{b}(M/x_dM) &=& \bigcap_{\underline{y}, j=1}^{d-1}
\mathrm{Ann}\,U_{M/(y_1,\ldots,y_{j-1},x_d)M}(0)\\
&=& \bigcap_{\underline{y}, j=1}^{d-1}
\mathrm{Ann}\,U_{M/(y_1,\ldots,y_{j-1},x_d^{n_d})M}(0)\\
 &=&\mathfrak{b}(M/x_d^{n_d}M),
\end{eqnarray*}
where $\underline{y} = y_1,\ldots,y_{d-1}$ runs over all systems of parameters of $M/x_dM$.\\
We now proceed by induction on $d$ that $\mathfrak{b}(M/(x_i, \ldots,x_d)M) =
\mathfrak{b}(M/(x_i^{n_i}, \ldots,x_d^{n_d})M)$. The case $d=2$
follows from the above fact since $i = 2$. Suppose $d \geq 3$ and
$i<d$. Applying the inductive hypothesis for $M/(x_{i+1},\ldots,x_d)M$ we have
$$\frak
b(M/(x_i,x_{i+1},\ldots,x_d)M) = \frak
b(M/(x_i^{n_i},x_{i+1},\ldots,x_d)M)$$ for all $n_i \geq 1$. By Lemma \ref{B3.1.10} we have $x_{i+1},\ldots,x_d$ is a part of $C$-system of parameters of $M/x_i^{n_1}M$. By using the inductive hypothesis for $M/x_i^{n_i}M$ we obtain
$$\frak b(M/(x_i^{n_i},x_{i+1},\ldots,x_d)M) =
\frak b(M/(x_i^{n_i},x_{i+1}^{n_{i+1}}\ldots,x_d^{n_d})M)$$ for all
$n_{i+1},\ldots,n_{d} \geq 1$. The proof is complete.
\end{proof}
\begin{corollary} \label{H3.2.12}
Let $\underline{x} = x_1, \ldots, x_d$ be a $C$-system of parameters of $M$. Then for all $d$-tuples of positive integers $\underline{n} =(n_1,\ldots,n_d)$ we have
$x_{1}^{n_1},
\ldots,x_d^{n_d}$ is also a $C$-system of parameters.
\end{corollary}
\begin{proof} The assertion follows immediately from Corollary \ref{H3.2.11} and the definition of $C$-system of parameters.
\end{proof}
\noindent {\bf An application to $dd$-sequences.} We use the Cohen-Macaulay deviated sequence to compute the function $I_{M,\underline{x}}(\underline{n})$.
\begin{proposition} \label{M3.2.13}
Let $\underline{x} = x_1, \ldots, x_d$ be a $C$-system of parameters of $M$. Let $\{U_i(M)\}_{i = 0}^{d-1}$  be the Cohen-Macaulay deviated sequence of $M$. Then the function
$$I_{M,\underline{x}}(\underline{n}) =
\ell(M/(x_1^{n_1},\ldots,x_d^{n_d})M) - n_1\ldots n_d e(x_1,\ldots,x_d;M)$$ is a
polynomial in $\underline{n} = n_1,\ldots,n_d$. More precisely, the equality
$$I_{M,\underline{x}}(\underline{n}) = \sum_{i=0}^{p(M)}n_1\ldots n_i e(x_1,\ldots,x_i;U_i(M))$$
holds for all $n_i \geq 1$, where $p(M)$ is the polynomial type of $M$. In particular,
$\underline{x} = x_1, \ldots, x_d$ is a $dd$-sequence system of parameters.
\end{proposition}
\begin{proof}
For all $d$-tuples of positive integers $\underline{n} = (n_1,\ldots,n_d)$ by
Corollary \ref{H3.2.12} we have $x_{1}^{n_1},
\ldots,x_d^{n_d}$ is a $C$-system of parameters. By Theorem \ref{D3.2.9} and Remark \ref{C3.2.2} (ii) we have
$${(x_{i+2}^{n_{i+2}},\ldots,x_d^{n_{d}})M:_M x_{i+1}^{n_{i+1}}}/{(x_{i+2}^{n_{i+2}},\ldots,x_d^{n_{d}})M} \cong U_i(M)$$
for all $0 \leq i \leq d-1$.
By the Auslander-Buchsbaum formula (cf. \cite[Corollary 4.3]{AB58}) we have
\begin{eqnarray*}
I_{M,\underline{x}}(\underline{n}) &=&
\sum_{i=0}^{d-1}e(x_{1}^{n_{1}},
\ldots,x_i^{n_i};{(x_{i+2}^{n_{i+2}},\ldots,x_d^{n_{d}})M:_M x_{i+1}^{n_{i+1}}}/{(x_{i+2}^{n_{i+2}},\ldots,x_d^{n_{d}})M})\\
&=& \sum_{i=0}^{d-1}e(x_{1}^{n_{1}},
\ldots,x_i^{n_i};U_i(M))\\
&=& \sum_{i=0}^{d-1}n_1\ldots n_i e(x_1,\ldots,x_i;U_i(M))
\end{eqnarray*}
is a polynomial in $n_1,\ldots,n_d$. By Remark \ref{C3.2.2}
(iii) we have $\mathrm{Ann}U_i(M) \supseteq \frak b(M)$ for all $i
\leq d-1$. Thus $\dim U_i \leq p(M)$ for all $i \leq d-1$ since $\dim
R/\frak b(M) = \dim R/\frak a(M) = p(M)$. Therefore $e(x_1,\ldots,x_i;U_i(M)) = 0$ for all $p(M) < i \leq d-1$. Hence, the equality 
$$I_{M,\underline{x}}(\underline{n}) = \sum_{i=0}^{p(M)}n_1\ldots n_i e(x_1,\ldots,x_i;U_i(M))$$
holds. The last assertion follows from Proposition \ref{M3.1.7}. The proof is complete.
\end{proof}
The following is a generalization of Proposition \ref{M3.1.7} (see also \cite[Theorem 3.7]{CN}).
\begin{corollary}
  Let $\underline{x} = x_1, \ldots, x_d$ be a $dd$-sequence system of parameters of $M$. Let $U_i(M)$, $0 \leq i \leq d-1$,  be the Cohen-Macaulay deviated sequence of $M$. Then the difference
$$I_{M,\underline{x}}(\underline{n}) = \sum_{i=0}^{p(M)}n_1\ldots n_i e(x_1,\ldots,x_i;U_i(M))$$
for all $n_i \geq 1$, where $p(M)$ is the polynomial type of $M$.
\end{corollary}
\begin{proof} Notice that if $\underline{x} = x_1, \ldots, x_d$ is a $dd$-sequence system of parameters of $M$, then $x_d^s$ is a $p$-standard parameter element for some $s \ge d$ (see Remark \ref{R dd seq}). So $x_d^s \in \mathfrak{a}(M)$, and hence $x_d^{k} \in \mathfrak{a}(M)^3 \subseteq \mathfrak{b}(M)^3$ for some $k \ge 3d$. Therefore, $x_d^k$ is a $C$-parameter element of $M$.
	Note that $x_1, \ldots, x_{d-1}$ is a $dd$-sequence of $M / x_d^k M$, so $x_{d-1}^k$ is a $C$-parameter element of $M / x_d^k M$ for all $k \ge 3d > 3(d-1)$.
	Continuing this process, we obtain that $x_1^k, \ldots, x_d^k$ form a $C$-system of parameters of $M$. So we have
$$I_{M,\underline{x}^k}(\underline{n}) = \sum_{i=0}^{p(M)}k^in_1\ldots n_i e(x_1,\ldots,x_i;U_i(M))$$
for all $n_i \geq 1$. By Proposition \ref{M3.1.7} we have
$$I_{M,\underline{x}}(kn_1, \ldots, kn_d) =
\sum_{i=0}^{p(M)} k^in_1\ldots n_i e(x_1,\ldots,x_i; 0:_{M/(x_{i+2},\ldots,x_d)M}x_{i+1})$$
for all $n_i \geq 1$. However it is clear that $I_{M,\underline{x}^k}(\underline{n}) = I_{M,\underline{x}}(kn_1, \ldots, kn_d)$. By the above equality we have
$$e(x_1,\ldots,x_i;U_i(M)) = e(x_1,\ldots,x_i; 0:_{M/(x_{i+2},\ldots,x_d)M}x_{i+1})$$
for all $i \le p(M)$. Therefore
$$I_{M,\underline{x}}(\underline{n}) = \sum_{i=0}^{p(M)}n_1\ldots n_i e(x_1,\ldots,x_i;U_i(M))$$
for all $n_i \geq 1$ by Proposition \ref{M3.1.7}. The proof is complete.
\end{proof}

\noindent{\bf Sequentially Cohen-Macaulay modules.} We give an application of the Cohen-Macaulay deviate sequence to characterize {\it sequentially Cohen-Macaulay} modules. This notion is firstly introduced by Stanley in the graded rings \cite{St96}, and for modules over local rings by Schenzel in \cite{Sch98}, and by Nhan and the first author in \cite{CN03}.
\begin{remark}[\cite{CC07-2}] \rm
\begin{enumerate}[{(i)}]
\item The filtration of submodules $\mathcal{D}: D_0 \subset D_1 \subset \cdots \subset D_t =M$ of $M$ is called {\it the dimension filtration} if $D_i = U_{D_{i+1}}(0)$ for all $i \leq t-1$.
\item We call $M$ {\it sequentially Cohen-Macaulay} module if $D_{i+1}/D_i$ is Cohen-Macaulay for all $i \leq t-1$.
\item A system of parameters $\underline{x} =
x_1,\ldots,x_d$ of $M$ is called {\it good} if $D_i \cap (x_{d_i+1},\ldots,x_d)M = 0$ for
$i= 0, 1, \ldots, t-1$, where $d_i = \dim D_i$ for all $i \leq t$. Notice that every $dd$-sequence system of parameters is good.
\end{enumerate}
\end{remark}

\begin{remark}\label{Q3.2.14} \rm Let $M$ be a finitely generated $R$-module of dimension $d$ with the dimension filtration
 $$\mathcal{D}: D_0 \subset D_1 \subset \cdots \subset D_t =M,$$
with $d_i = \dim D_i $ for all $i \leq t$. Let $\underline{x} = x_1,
\ldots, x_d$ be a $C$-system of parameters of $M$. For each
$i<t$ and $d_i \leq j \leq d-1$ we have
$$D_i \cap (x_{j+2},\ldots,x_d)M = 0.$$
Therefore we can identify $D_i$ with a submodule of
$M/(x_{j+2},\ldots,x_d)M$. Moreover, since $\dim D_i = d_i < j+1 = \dim
M/(x_{j+2},\ldots,x_d)M$, $D_i$ is isomorphic to a submodule of
$U_j(M)$ for all $d_i \leq j \leq d-1$. So without any confusion we write $D_i \subseteq U_j(M)$ for all $d_i \leq j \leq d-1$.
\end{remark}
The following is a characterization of sequentially Cohen-Macaulay modules.
\begin{proposition}\label{M3.2.15} Let $M$ be a finitely generated $R$-module of dimension $d$ with the dimension filtration
 $$\mathcal{D}: D_0 \subset D_1 \subset \cdots \subset D_t =M,$$
with $d_i = \dim D_i $ for all $i \leq t$. Let $U_i(M)$, $0 \leq i \leq d-1$, be the Cohen-Macaulay deviated sequence of $M$. The following statements are equivalent
\begin{enumerate}[{(i)}]\rm
\item {\it $M$ is a sequentially Cohen-Macaulay module.}
\item {\it $D_i = U_j(M)$ for all $i<t$ and for all $d_i \leq j <d_{i+1}$.}
\end{enumerate}
\end{proposition}
\begin{proof} $\mathrm{(i)}\Rightarrow \mathrm{(ii)}$ Let $\underline{x} = x_1,
\ldots, x_d$ be a $C$-system of parameters of $M$. By Proposition \ref{M3.2.13} it is a $dd$-sequence. By \cite[Lemma 6.4]{CC07-1}, $M/(x_{j+2},\ldots,x_d)M$ is a sequentially Cohen-Macaulay module with the dimension filtration \\
$$ D_0 \cong \frac{D_0+(x_{j+2},\ldots,x_d)M}{(x_{j+2},\ldots,x_d)M} \subset  \cdots  \subset D_i\cong \frac{D_i+(x_{j+2},\ldots,x_d)M}{(x_{j+2},\ldots,x_d)M} \subset M/(x_{j+2},\ldots,x_d)M$$
for all $i<t$ and for all $d_i \leq j <d_{i+1}$. Thus $D_i = U_j(M)$ for all $i<t$ and for all $d_i \leq j <d_{i+1}$.\\
$\mathrm{(ii) }\Rightarrow \mathrm{(i)}$ Let $\underline{x} = x_1,
\ldots, x_d$ be a $C$-system of parameters of $M$. By Proposition
\ref{M3.2.13} we have
$$I_{M,\underline{x}}(\underline{n}) = \sum_{j=0}^{d-1}n_1\ldots n_j e(x_1,\ldots,x_j;U_j(M))$$
for all $n_1,..,n_d \geq 1$.
Since $D_i = U_j(M)$ for all $i<t$ and for all $d_i \leq j <d_{i+1}$ we have
$e(x_1,\ldots,x_j;U_j(M)) = 0$ for all $i<t$ and for all $d_i < j <d_{i+1}$. Therefore
$$I_{M,\underline{x}}(\underline{n}) = \sum_{i=0}^{t-1}n_1\ldots n_{d_i} e(x_1,\ldots,x_{d_i};D_i)$$
for all $n_1,..,n_d \geq 1$. Hence $M$ is a sequentially Cohen-Macaulay module by \cite[Theorem 4.2]{CC07-2}. The proof is complete.
\end{proof}

\noindent{ \bf Relation with the Serre condition ($S_2$).} For each $R$-module $M$ we have a set of invariant modules $U_i(M)$, $0 \leq i \leq d-1$, as Theorem \ref{D3.2.9}. Therefore we have a special set of prime ideals, $\cup_{i=0}^{d-1}\mathrm{Ass}\, U_i(M)$, determined by $M$. If $\frak p \in
\mathrm{Ass}\, M$ and $\dim R/\frak p < d$, then $\frak p \in
\mathrm{Ass}\, U_M(0) = \mathrm{Ass}\, U_{d-1}(M)$. In the following we consider the relation between $\mathrm{Ass}\, U_{d-2}(M)$ and {\it the Serre condition
$(S_2)$}.
\begin{definition}\rm  For all $n \geq 1$, we say that $M$ satisfies {\it the Serre condition
$(S_n)$} at the prime ideal $\frak p \in \mathrm{Supp}(M)$ if
$$\mathrm{depth}M_\frak p \geq \min \{\dim M_{\frak p}, n\}.$$
Moreover, $M$ has property $(S_n)$ if it satisfies the Serre condition $(S_n)$ at all $\frak p \in
\mathrm{Supp}(M)$.
\end{definition}

It is obvious that $R$ satisfies the condition $(S_1)$ if and only if
$\mathrm{Ass}\, R = \mathrm{minAss}R$. Furthermore, if $R$ satisfies the condition
$(S_2)$ and $R$ is {\it catenary} (this condition is always true if $R$
is the homomorphic image of a Cohen-Macaulay ring), then $\mathrm{Ass}\, R =
\mathrm{Assh}\,R$ (see \cite[Corollary
2.24]{Sch98-1}). Conversely, Goto and Nakamura \cite[Lemma
3.2]{GN01} proved that if $\mathrm{Ass}\, R \subseteq
\mathrm{Assh}\,R \cup \{\frak m\}$, then the set
$$\mathcal{F}(R) = \{\mathfrak{p}\in \mathrm{Spec}(R)\,|\, \dim R_{\mathfrak{p}}>
1=\mathrm{depth}R_{\mathfrak{p}},\, \mathfrak{p} \neq \mathfrak{m}
\}$$ is finite, i.e. $R$ does not satisfy the Serre condition $(S_2)$ at only finitely many prime ideals. The set $\mathcal{F}(R)$ can be described as follows.
\begin{proposition}\label{M3.2.17}
Suppose that $\mathrm{Ass}\, M \subseteq \mathrm{Assh}\,M
\cup \{\frak m\}$. Set
$$\mathcal{F}(M) = \{\mathfrak{p}\in \mathrm{Supp}(M)\,|\, \dim M_{\mathfrak{p}}>
1=\mathrm{depth}\,M_{\mathfrak{p}},\, \mathfrak{p} \neq \mathfrak{m}
\}.$$ Then $\mathcal{F}(M) = \mathrm{Ass}\, U_{d-2}(M) \setminus
\{\frak m\}$.
\end{proposition}
\begin{proof} Notice that $R$ is the homomorphic image of a Cohen-Macaulay local ring, and $M$ is equidimensional since $\mathrm{Ass}\, M \subseteq \mathrm{Assh}\, M \cup \{\frak m\}$. Let $x$ be a $C$-parameter element of $M$. For all $\frak p \in \mathrm{Ass}\, U_{d-2}(M) \setminus \{\frak
m\}$ we have $\frak p \in \mathrm{Ass}\, M/xM$ and $\dim R/\frak p \leq
d-2$. Hence $\dim M_{\mathfrak{p}}>
1=\mathrm{depth}\,M_{\mathfrak{p}}$. So $\mathrm{Ass}\, U_{d-2}(M)
\setminus \{\frak m\} \subseteq \mathcal{F}(M)$.\\
Conversely, let $\mathfrak{p}\in \mathcal{F}(M)$. Since
$\mathrm{depth}\,M_{\mathfrak{p}} = 1$, for every parameter element $z
\in \mathfrak{p}$ we have $\mathfrak{p}\in \mathrm{Ass}\,  M/zM$. Therefore
$\mathfrak{p}\in \mathrm{Ass}\,  M/(xz)M$. Notice that $xz$ is a $C$-parameter element of $M$ and $\dim R/\frak p \leq
d-2$, so $\frak p \in \mathrm{Ass}\, U_{M/(xz)M}(0) \cong
\mathrm{Ass}\, U_{d-2}(M)$. The proof is complete.
\end{proof}

\begin{remark}\label{C3.2.18}\rm Let $M$ be a finitely generated $R$-module.
\begin{enumerate}[{(i)}]
\item Suppose that $\mathrm{Ass}\, M \subseteq \mathrm{Assh}\,M \cup
\{\frak m\}$ and $\mathcal{F}(M)$ as the previous proposition. Let $x$ be a parameter element of $M$ such that $x
\notin \frak p$ for all $\frak p \in \mathcal{F}(M)$. Then $M$
satisfies the Serre condition $(S_2)$ at all prime ideals $\frak p \in \mathrm{Supp}\,M$ containing $x$ and $\frak p
\neq \frak m$. So $M/xM$ satisfies the Serre condition $(S_1)$ at all $\frak p \in \mathrm{Supp}\,(M/xM)$ and $\frak p
\neq \frak m$. Hence
$$\mathrm{Ass}\, (M/xM) \subseteq \mathrm{minAss}\,(M/xM) \cup \{\frak m\} = \mathrm{Assh}\,(M/xM) \cup \{\frak m\}.$$
\item Set $\overline{M} = M/U_M(0)$. Let $x \in \frak b(M)^3 \cap \frak
b(\overline{M})^3$ be a parameter element of $M$ and hence of $\overline{M}$. Put $\frak b' = \frak b(M/xM)$,
$\frak b'' = \frak b(\overline{M}/x\overline{M})$ and $\frak b  =
\frak b' \cap \frak b''$. We have that $\dim R/\frak b \leq d-2$. By Remark \ref{C3.2.2} (i) we have
$U_{d-2}(M) \cong H^0_{\frak b'}(M/xM) \subseteq H^0_{\frak
b}(M/xM)$. However $\dim H^0_{\frak b}(M/xM) < d-1$, so
$U_{d-2}(M) \cong H^0_{\frak b}(M/xM)$. Similarly, we have
$U_{d-2}(\overline{M}) \cong H^0_{\frak
b}(\overline{M}/x\overline{M})$. By the proof of Lemma \ref{B3.2.8} we have
$$U_{d-2}(M) \cong H^0_{\frak
b}(M) \oplus H^1_{\frak b}(\overline{M})$$ and
$$U_{d-2}(\overline{M}) \cong H^0_{\frak
b}(\overline{M}) \oplus H^1_{\frak b}(\overline{M}) =
H^1_{\frak b}(\overline{M}).$$
Thus
$U_{d-2}(\overline{M})$ is isomorphism to a direct summand of $U_{d-2}(M)$.
\end{enumerate}
\end{remark}

The next result will play an important role in the next section.
\begin{proposition}\label{M3.2.19} Let $M$ be a finitely generated $R$-module of dimension $d \geq
2$. Let $x$ be a parameter element of $M$ such that $x \notin \frak p$ for all $\frak p \in \big( \mathrm{Ass}\, U_{M}(0)  \cup \mathrm{Ass}\, U_{d-2}(M) \big)
\setminus \{\frak m\}$. Then we have the following short exact sequence
$$0 \to U_M(0)/xU_M(0) \to U_{M/xM}(0) \to H^0_{\frak m}(\overline{M}/x\overline{M}) \to 0,$$
where $\overline{M} = M/U_M(0)$.
\end{proposition}
\begin{proof}
Since $U_M(0) \cap xM = x(U_M(0) :_Mx) = xU_M(0)$, we have the following short exact sequence
$$0 \to U_M(0)/xU_M(0) \overset{\varphi}{\to} M/xM \to \overline{M}/x\overline{M} \to 0.$$
If $\dim U_M(0) = 0$ then $\dim U_M(0)/xU_M(0) < d-1$. If $\dim
U_M(0) > 0$ then $x$ is a parameter element of both $M$ and $U_M(0)$ so
$\dim U_M(0)/xU_M(0) = \dim U_M(0) - 1 < d-1$. Notice that $\mathrm{Im}(\varphi) = (U_M(0) + xM)/xM$. Thus we always have that
$(U_M(0) + xM)/xM$ is a submodule of $M/xM$ of dimension less than $d-1$. Hence $\mathrm{Im}(\varphi) = (U_M(0) + xM)/xM \subseteq U_{M/xM}(0)$. So we have the short exact sequence
$$0 \to U_M(0)/xU_M(0) \to U_{M/xM}(0) \to U_{\overline{M}/x\overline{M}}(0) \to 0.$$
On the other hand $x \notin \frak p$ for all $\frak p \in
\mathrm{Ass}\, U_{d-2}(M) \setminus \{\frak m\}$. So $x \notin \frak p$
for all $\frak p \in \mathrm{Ass}\, U_{d-2}(\overline{M}) \setminus
\{\frak m\}$ by Remark \ref{C3.2.18} (ii). By Remark \ref{C3.2.18} (i) we have
$$\mathrm{Ass}\, (\overline{M}/x\overline{M}) \subseteq
\mathrm{Assh}\,(\overline{M}/x\overline{M}) \cup \{\frak m\}.$$ Therefore
$U_{\overline{M}/x\overline{M}}(0) = H^0_{\frak
m}(\overline{M}/x\overline{M})$. Thus we obtain the short exact sequence
$$0 \to U_M(0)/xU_M(0) \to U_{M/xM}(0) \to H^0_{\frak m}(\overline{M}/x\overline{M}) \to 0.$$
\end{proof}

\section{The unmixed degrees}
In this section let $I$ be an $\frak m$-primary ideal and $M$ a finitely generated
$R$-module of dimension $d
> 0$.  Let $U_i(M)$, $0 \leq i \leq d-1$, be the Cohen-Macaulay deviated sequence of $M$. The purpose of this section is to construct a new degree for
$M$ in terms of $U_i(M)$. Firstly, recalling that
the length function $\ell(M/I^nM)$ becomes a polynomial of degree $d$ when $n \gg 0$ and
$$\ell(M/I^{n+1}M) = \sum_{i=0}^d(-1)^i e_i(I,M) \binom{n+d-i}{d-i}.$$
The coefficients $e_i(I, M)$, for $i = 0,\ldots,d$, are called
the Hilbert coefficients of $M$ with respect to $I$. Especially the leading coefficient
$e_0(I, M)$ or $e(I, M)$ is called {\it the Hilbert-Samuel multiplicity} of $M$ with respect to $I$. If $I=\frak m$, the multiplicity is denoted by $e(M)$ for simplicity. In this section we denote by $\mathrm{deg}(I, M)$ (resp.
$\mathrm{deg}(M)$) the multiplicity $e(I, M)$ (resp. $e(M)$) and call it {\it the degree} of $M$ with respect to $I$ (resp. the degree of $M$). The following associativity formula for the degree says that $\mathrm{deg}(I, M)$ depends only on the associated prime ideals of the highest dimension
(see \cite[Corollary 4.7.8]{BH98})
$$\mathrm{deg}(I, M) = \sum_{\frak p \in \mathrm{Assh}\,M}\ell_{R_{\frak p}}(M_{\frak p})\, \mathrm{deg}(I, R/\frak p). $$
Notice that if $\frak p \in \mathrm{minAss}\,M$, then $M_{\frak p}$ has finite length and
$M_{\frak p} = H^0_{\frak p R_{\frak p}}(M_{\frak p})$. So we have
$$\mathrm{deg}(I, M) = \sum_{\frak p \in \mathrm{Assh}\,M}
\ell_{R_{\frak p}}(H^0_{\frak p R_{\frak p}}(M_{\frak p}))\,
\mathrm{deg}(I, R/\frak p). $$ We next recall some other degrees of $M$ related to $\mathrm{deg}(I, M)$
(see \cite{V98-2}).
\begin{definition}\label{adeg}\rm
The {\it arithmetic degree} of $M$ with respect to $I$, denoted by $\mathrm{adeg}(I, M)$, is the integer
$$\mathrm{adeg}(I, M) =
\sum_{\frak p \in \mathrm{Ass}\, M}\ell_{R_{\frak p}}(H^0_{\frak p
R_{\frak p}}(M_{\frak p}))\, \mathrm{deg}(I, R/\frak p). $$
\end{definition}

\begin{remark} \label{C3.3.2} \rm
\begin{enumerate}[{(i)}]
\item Let $\mathcal{D} : D_0 \subseteq D_1 \subseteq \cdots
\subseteq D_t = M$ be the dimension filtration of $M$, then we have $\mathrm{adeg}(I, M)
= \sum_{i=0}^t \mathrm{deg}(I, D_i)$. So $\mathrm{adeg}(I, M) \geq \mathrm{deg}(I,
M)$ and the equality occurs if and only if $U_M(0) = 0$.
\item  Moreover, assume $(R, \frak m)$ is the homomorphic image of a Gorenstein local ring $(S, \frak n)$ of dimension
$n$. Then $\mathrm{adeg}(I, M)$ can be determined without the knowledge of the primary decomposition as follows
$$\mathrm{adeg}(I, M) = \sum_{i} \mathrm{deg}(I, \mathrm{Ext}^{i}_S(\mathrm{Ext}^{i}_S(M, S), S)).$$
\end{enumerate}
\end{remark}

Vasconcelos et al. \cite{DGV98, V98-1, V98-2} introduced the notion of {\it extended degree of graded modules} in order to capture the size of a module along with some of the complexity of its
structure. The prototype of an extended degree is the {\it homological degree} introduced and studied by Vasconselos in  \cite{V98-1} (see also \cite{V98-2}). The extended degree for local rings was considered by Rossi, Trung and Valla in \cite{RTV03}.
This notion is associated with an $\frak m$-primary ideal $I$ in \cite{L05}. From this definition we always assume that $(R, \frak m)$ is a local ring with infinite residue field.

\begin{definition}\label{D3.3.3}\rm Let $\mathcal{M}(R)$ be the category of finitely generated
$R$-modules. An {\it extended degree} on $\mathcal{M}(R)$
with respect to $I$ is a numerical function
$$ \mathrm{Deg}(I, \bullet) : \mathcal{M}(R) \to \mathbb{R} $$
satisfying the following conditions
\begin{enumerate}[{(i)}]
\item $\mathrm{Deg} (I, M) = \mathrm{Deg}(I, \overline{ M}) + \ell(H^0_{\frak
m}(M))$, where $\overline{M} = M/H^0_{\frak m}(M)$;
\item (Bertini's rule) $\mathrm{Deg}(I, M) \geq \mathrm{Deg}(I, M/xM)$ for every generic element $x \in I\setminus \frak mI$ of $M$;
\item If $M$ is Cohen-Macaulay then $\mathrm{Deg}(I, M) =
\mathrm{deg}(I, M)$.
\end{enumerate}
\end{definition}

The homological degree is a typical extended degree that is defined as follows.
\begin{definition}[\cite{V98-1}] \label{D3.3.4} \rm Assume that $(R, \frak m)$ is the homomorphic image of a Gorenstein local ring
$(S, \frak n)$ of dimension $n$, and $M$ a finitely generated $R$-module of dimension $d$. Then the {\it homological degree}, $\mathrm{hdeg}(I, M)$, of $M$ with respect to $I$ is defined by the following recursive formula
$$\mathrm{hdeg}(I, M) = \mathrm{deg}(I, M) + \sum_{i=n-d+1}^n \binom{d-1}{i-n+d-1} \mathrm{hdeg}(I, \mathrm{Ext}^i_S(M, S)).$$
\end{definition}

\begin{remark} \label{C3.3.5} \rm
\begin{enumerate}[{(i)}]
\item The Definition \ref{D3.3.4} is recursive on dimension since $\dim \mathrm{Ext}^i_S(M, S) \le n-i <
d$ for all $i = n-d+1,\ldots, n$.
\item The homological degree $\mathrm{hdeg}(I,\bullet)$ is an extended degree on $\mathcal{M}(R)$, and $\mathrm{hdeg}(I,M) =
\mathrm{deg}(I, M)$ if and only if $M$ is Cohen-Macaulay.
\item If $M$ is a generalized Cohen-Macaulay module, then $\ell(\mathrm{Ext}^{n-i}_S(M,
S)) = \ell(H^i_{\frak m}(M))$ for all $i = 0, \ldots, d-1$ by the local duality theorem. We have
$$\mathrm{hdeg}(I, M) = \mathrm{deg}(I, M) + \sum_{i=0}^{d-1} \binom{d-1}{i} \ell(H^i_{\frak m}(M)).$$
\item (\cite[Proposition 3.5]{V98-2}) If $\dim M = \dim S =2$ then
$$\mathrm{hdeg}(I, M) = \mathrm{adeg}(I, M) + \ell(\mathrm{Ext}^2_S(\mathrm{Ext}^1_S(M, S),S)).$$
\end{enumerate}
\end{remark}

The purpose of this section is to introduce another extended degree on $\mathcal{M}(R)$ in terms of the Cohen-Macaulay deviated sequence $U_i(M)$, $i = 0,\ldots,d-1$. Notice that $\dim U_i(M) \leq i$ for all $0 \leq
i \leq d-1$.
\begin{definition}\rm Let $M$ be a finitely generated $R$-module of dimension
$d$ and $U_i(M)$, $0 \leq i \leq d-1$, the Cohen-Macaulay deviated sequence of $M$. We define the {\it unmixed degree} of $M$ with respect to
$I$, $\mathrm{udeg}(I, M)$, as follows
$$\mathrm{udeg}(I, M) = \mathrm{deg}(I, M) + \sum_{i=0}^{d-1}\delta_{i, \dim U_i(M)}\mathrm{deg}(I, U_i(M)),$$
where $\delta_{i, \dim U_i(M)}$ is the Kronecker symbol.
\end{definition}
It is worth noting that in the above definition and Proposition \ref{M3.2.13} we consider the subsequence of modules of the Cohen-Macaulay deviated sequence consisting of $U_i(M)$ with $\dim U_i(M) = i$. We call this subsequence the {\it reduced Cohen-Macaulay deviated sequence} of $M$. In the rest of this paper, we shall prove that the unmixed degree is an extended degree. The first condition of Definition \ref{D3.3.3} follows from the following.

\begin{proposition}\label{M3.3.9}
Let $N$ be a submodule of finite length of $M$. Then
$$\mathrm{udeg}(I,M) = \mathrm{udeg}(I,M/N) + \ell(N).$$
\end{proposition}
\begin{proof}
Let $x_1,\ldots,x_d$ be a $C$-system of parameters of both $M$ and $M/N$. By Proposition \ref{M3.2.13}
$x_1,\ldots,x_d$ is a $dd$-sequence of $M$. So $H^0_{\frak m}(M) \cap
(x_1,\ldots,x_d)M = 0$. For all $0 \leq j \leq d-1$, we have the short exact sequence
$$0 \to N \to M/(x_{j+2},\ldots,x_d)M \to M/(N+(x_{j+2},\ldots,x_d)M) \to 0.$$
Therefore $U_j(M/N) \cong U_j(M)/N$ for all $0 \leq j \leq d-1$. Thus
$$\delta_{j, \dim U_j(M/N)} \mathrm{deg}(I,U_j(M/N)) =
 \delta_{j, \dim U_j(M)} \mathrm{deg}(I,U_j(M))$$
for all $1 \leq j \leq d-1$ and
$$\delta_{0, \dim U_0(M/N)} \mathrm{deg}(I,U_0(M/N)) =
\delta_{0, \dim U_0(M)} \mathrm{deg}(I,U_0(M)) - \ell(N).$$ The claim is now obvious.
\end{proof}
The next result shows that $\mathrm{udeg}(I, M)$ agrees with $\mathrm{hdeg}(I, M)$ for generalized Cohen-Macaulay modules.
\begin{proposition} \label{M3.3.10} Let $M$ be a generalized Cohen-Macaulay $R$-module of dimension $d$. Then
$$\mathrm{udeg}(I,M) = \mathrm{deg}(I,M) + \sum_{j=0}^{d-1} \binom{d-1}{j}\ell(H^j_\frak{m}(M)).$$
\end{proposition}
\begin{proof}
Let $x_1,\ldots,x_d$ be a $C$-system of parameters of $M$. By Corollary \ref{H3.2.5} (see also \cite[Corollary 4.2]{CQ11}) we have
$$U_i(M) \cong H^0_{\mathfrak{m}}(M/(x_{i+2},\ldots,x_d)M) \cong
\bigoplus_{j=0}^{d-i-1} H^j_{\mathfrak{m}}(M)^{\binom{d-i-1}{j}},$$
for all $0 \leq i \leq d-1$. So $\dim U_i(M) = 0$ for all $i \le d-1$. Therefore $\delta_{i, \dim U_i(M)}\mathrm{deg}(I,U_i(M))
= 0$ for all $1 \leq i \leq d-1$, and
$$\delta_{0, \dim U_0(M)}\mathrm{deg}(I,U_0(M)) = \sum_{j=0}^{d-1}
\binom{d-1}{j}\ell(H^j_\frak{m}(M)).$$
\end{proof}
We next compute the unmixed degree when $\dim M$ is small.
\begin{proposition}\label{M3.3.11} The following statements hold true.
\begin{enumerate}[{(i)}]\rm
\item {\it If $d = 1$, then the equality $\mathrm{udeg}(I,M) = \mathrm{adeg}(I,M)$ holds.}
\item {\it If $d = 2$, then the equality $\mathrm{udeg}(I,M) = \mathrm{adeg}(I,M) + \ell(H^1_{\frak m}(M/U_M(0)))$ holds.}
\end{enumerate}
\end{proposition}
\begin{proof} (i) It is clear.\\
(ii) We consider the following two cases.\\
The case $\dim U_M(0) = 0$, we have $M$ is a generalized
Cohen-Macaulay module. Therefore, by Proposition \ref{M3.3.10} we have
\begin{eqnarray*}
\mathrm{udeg}(I,M) &=& \mathrm{deg}(I,M) + \ell(H^0_{\frak m}(M)) +
\ell(H^1_{\frak m}(M))\\
&=& \mathrm{adeg}(I,M) + \ell(H^1_{\frak m}(M/H^0_{\frak m}(M))).
\end{eqnarray*}
The case $\dim U_M(0) = 1$. Consider the dimension filtration $H^0_{\frak m}(M) \subset U_M(0) \subset M$ of $M$. By Remark \ref{C3.3.2} (i) we have
$$\mathrm{adeg}(I,M) =
\mathrm{deg}(I,M) + \mathrm{deg}(I,U_M(0)) + \ell(H^0_{\frak
m}(M)).$$ On the other hand $U_1(M) \cong U_M(0)$ so
$\delta_{1, \dim U_1(M)}\mathrm{deg}(I,U_1(M)) = \mathrm{deg}(I,U_M(0)) $. Let
$x_2$ be a $C$-parameter element of $M$. By Corollary \ref{H3.2.5} we have
$$U_0(M) \cong H^0_{\frak m}(M/x_2M) \cong  H^0_{\frak m}(M) \oplus H^1_{\frak m}(M/U_M(0)).$$
Thus $\delta_{0, \dim U_0(M)}\mathrm{deg}(I,U_0(M)) = \ell (H^0_{\frak m}(M)) + \ell (H^1_{\frak m}(M/U_M(0)))$. Therefore, the equality
$$\mathrm{udeg}(I,M) = \mathrm{adeg}(I,M) + \ell(H^1_{\frak m}(M/U_M(0)))$$
holds. The proof is complete.
\end{proof}
\begin{corollary} Suppose $(R, \frak m)$ is the homomorphic image of a
Gorenstein local ring and $\dim M = 2$. Then $\mathrm{udeg}(I,M)  =
\mathrm{hdeg}(I,M) $.
\end{corollary}
\begin{proof} Without loss of generality we may assume that $(R, \frak
m)$ is a Gorenstein local ring of dimension two. If $U_M(0) = H^0_{\frak
m}(M)$ we have that $M$ is generalized Cohen-Macaulay, and the claim follows from Proposition \ref{M3.3.10} and  Remark \ref{C3.3.5} (iii). Suppose $\dim
U_M(0) = 1$, by Proposition \ref{M3.3.11} and Remark \ref{C3.3.5} (iv)
we need only to show that
$$\ell (H^1_{\frak m}(M/U_M(0))) =
\ell(\mathrm{Ext}^2_R(\mathrm{Ext}^1_R(M, R),R)).$$
Since
$\mathrm{Ass}\, M/U_M(0) = \{\frak p \mid \frak p\in \mathrm{Ass}\, M,
\dim R/\frak p = 2 \}$ we have that
$\mathrm{Ext}^1_R(M/U_M(0),R)$ is a module of finite length, and
$\ell(\mathrm{Ext}^1_R(M/U_M(0),R)) = \ell(H^1_{\frak m}(M/U_M(0)))$
by the local duality theorem. By the local duality theorem again we have
$\ell(\mathrm{Ext}^2_R(\mathrm{Ext}^1_R(M, R),R)) = \ell
(H^0_{\frak m}(\mathrm{Ext}^1_R(M, R)))$. So it is enough to prove that
$$\ell(\mathrm{Ext}^1_R(M/U_M(0),R)) = \ell (H^0_{\frak
m}(\mathrm{Ext}^1_R(M, R))).$$
Indeed, consider the short exact sequence
$$0 \to U_M(0) \to M \to M/U_M(0) \to 0.$$
Since $\dim U_M(0) = 1$ and $\mathrm{depth}M/U_M(0)>0$, $\mathrm{Hom}_R(U_M(0),R) = \mathrm{Ext}^2_R(M/U_M(0),R) = 0$. So we have the following short exact sequence
$$0 \to \mathrm{Ext}^1_R(M/U_M(0),R) \to \mathrm{Ext}^1_R(M,R) \to \mathrm{Ext}^1_R(U_M(0),R) \to 0.$$
By \cite[Lemma 1.9]{Sch98-1} (v) we have that $\mathrm{Ext}^1_R(U_M(0),R)$ is $(S_2)$, and hence it is a Cohen-Macaulay module of dimension one. Thus  $H^0_{\frak
m}(\mathrm{Ext}^1_R(U_M(0),R)) = 0$. Therefore
$$\mathrm{Ext}^1_R(M/U_M(0),R)) = H^0_{\frak
m}(\mathrm{Ext}^1_R(M/U_M(0),R))) \cong H^0_{\frak
m}(\mathrm{Ext}^1_R(M, R)).$$ The proof is complete.
\end{proof}
In the following we prove the third condition of Definition \ref{D3.3.3}. Moreover, we also give a characterization of sequentially Cohen-Macaulay modules in terms of unmixed degrees.
\begin{theorem}\label{D3.3.8} Let $M$ be a finitely generated $R$-module of dimension
$d$. We have inequalities
$$\mathrm{deg}(I, M) \leq \mathrm{adeg}(I, M) \leq \mathrm{udeg}(I, M).$$
Furthermore
\begin{enumerate}[{(i)}]\rm
\item {\it The equality $\mathrm{deg}(I, M) = \mathrm{udeg}(I, M)$ holds if and only if $M$ is a Cohen-Macaulay module.}
\item {\it The equality $\mathrm{adeg}(I, M) = \mathrm{udeg}(I, M)$ holds if and only if $M$ is a sequentially
Cohen-Macaulay module.}
\end{enumerate}
\end{theorem}
\begin{proof} The first inequality is clear. Let
 $$\mathcal{D}: D_0 \subset D_1 \subset \cdots \subset D_t =M$$
be the dimension filtration of $M$ with $d_i = \dim D_i $ for all $i \leq t$. Recall that
$$\mathrm{adeg}(I,M) = \mathrm{deg}(I,M) + \sum_{i=0}^{t-1}\mathrm{deg}(I,D_i).$$
For all $i<t$ by Remark \ref{Q3.2.14} we have $D_i \subseteq
U_{d_i}(M)$. So $\dim U_{d_i}(M) = d_i $ and then
$$\mathrm{deg}(I,D_i) \leq
\mathrm{deg}(I,U_{d_i}(M)) = \delta_{d_i, \dim U_{d_i}(M)}\mathrm{deg}(I,
U_{d_i}(M)).$$
Thus $\mathrm{adeg}(I,M) \leq \mathrm{udeg}(I,M)$. \\
We have (i) follows from (ii), so it is enough to prove (ii). If $M$ is sequentially Cohen-Macaulay, then by Proposition
\ref{M3.2.15} we have that $\mathrm{adeg}(I,M) = \mathrm{udeg}(I,M)$.\\
Conversely, suppose
$\mathrm{adeg}(I,M) = \mathrm{udeg}(I,M)$. We have
$$\mathrm{deg}(I,D_i) = \mathrm{deg}(I,U_{d_i}(M)) $$
for all $i<t$, and
$$\delta_{j, \dim U_j(M)}\mathrm{deg}(I,U_j(M)) = 0 \quad \quad \quad \quad \quad \quad (4)$$ for all $i< t$ and $d_i < j
< d_{i+1}$. Let $\underline{x} = x_1, \ldots, x_d$ be a $C$-system of parameters of $M$. Thus
$$e(x_1, \ldots, x_{d_i}; D_i) = \mathrm{deg}((\underline{x}), D_i) = \mathrm{deg}((\underline{x}), U_{d_i}(M)) = e(x_1, \ldots, x_{d_i}; U_{d_i}(M))$$
for all $i < t$. By $(4)$ we have $\dim U_j(M)<j$  for all $d_i < j
< d_{i+1}$ and $i< t$, so 
$$e(x_1, \ldots, x_j; U_j(M)) = 0$$
for all $d_i < j
< d_{i+1}$ and $i< t$. By Proposition \ref{M3.2.13} we have
$$I_{M,\underline{x}}(\underline{n}) = \sum_{j=0}^{d-1}n_1\ldots n_j e(x_1,\ldots,x_j;U_j(M)).$$
for all $n_1, \ldots, n_d \ge 1$. Thus we have
$$I_{M,\underline{x}}(\underline{n}) = \sum_{i=0}^{t-1}n_1\ldots n_{d_i} e(x_1,\ldots,x_{d_i};D_i).$$
for all $n_1, \ldots, n_d \ge 1$. Hence $M$ is a sequentially Cohen-Macaulay module by \cite[Theorem 4.2]{CC07-2}. 
The proof is complete.
\end{proof}
In order to prove the Bertini rule of Definition \ref{D3.3.3}, we will show that the unmixed degree has good behavior by passing to the quotient modules regarding certain {\it superficial} elements.
\begin{definition}\rm An element $x \in I \setminus \frak
mI$ is called a {\it superficial} element of $M$ with respect to  $I$ if
there exists a positive integer $c$ such that
$$(I^{n+1}M:x) \cap I^cM = I^nM$$
for all $n \geq c$.
\end{definition}
\begin{remark}\label{C3.3.14}\rm
\begin{enumerate}[{(i)}]
\item Let $G_I(R) = \oplus_{n \geq 0} I^n/I^{n+1}$ be the associated graded ring of $R$ with respect to $I$, and $G_I(M)=
\oplus_{n \geq 0} I^nM/I^{n+1}M$ the graded $G_I(R)$-module. Set
$(G_I(R))_+ = \oplus_{n \geq 1} I^n/I^{n+1}$. Then $x$ is a superficial element of $M$ with respect to $I$ if and only if the {\it initial}
$x^*$ of $x$ in $G_I(R)$ is a $(G_I(R))_+$-filter regular element of
$G_I(M)$, i.e. $\ell (0:_{G_I(M)}x^*) < \infty$ (notice that in our context $I$ is $\frak m$-primary). Moreover, if $x$ is a superficial element, then it is an $I$-filter regular element of $M$.
\item A superficial element of $M$
with respect to $I$ always exists if the residue field $R/\frak m$ is infinite, a hypothesis which never cause us any problem because we can replace $R$ by the local ring
$R[X]_{\frak mR[X]}$, where $X$ is an indeterminate.
\item (cf. \cite[22.6]{N62}) Let $x$ be a superficial element of $M$
with respect to $I$. For $n \gg 0$ we have $I^{n+1}M:_Mx = 0:_Mx + I^n
M$ so
$$\ell(M/(I^{n+1}+(x))M) = \ell(M/I^{n+1}M) - \ell(M/I^nM) + \ell(0:_Mx)$$
for all $n \gg 0$.
\item Let $x$ be a superficial element of $M$
with respect to $I$. By (iii) we have $\mathrm{deg}(I,M/xM) =
\mathrm{deg}(I,M) $ if $d \geq 2$, and $\ell(M/xM) =
\mathrm{deg}(I,M/xM) = \mathrm{deg}(I,M) + \ell(0:_Mx)$ if $d=1$.
\end{enumerate}
\end{remark}

We need some lemmas before proving the Bertini rule of unmixed degrees.
\begin{lemma}\label{B3.3.15} Let $M$ be a finitely generated $R$-module of dimension $d \geq
2$. Let $x$ be a parameter element of $M$ such that $x$ is a superficial element of
$U_M(0)$ with respect to $I$ and $x \notin \frak p$ for all
$\frak p \in \mathrm{Ass}\, U_{d-2}(M) \setminus \{\frak m\}$. Then
$$\delta_{d-2, \dim U_{M/xM}(0)}\mathrm{deg}(I,U_{M/xM}(0)) = \delta_{d-1, \dim U_M(0)}\mathrm{deg}(I,U_{M}(0))$$
if $d \geq 3$, and
$$\delta_{0, \dim U_{M/xM}(0)}\mathrm{deg}(I,U_{M/xM}(0)) = \delta_{1, \dim U_M(0)}\mathrm{deg}(I,U_{M}(0)) +
\ell(0:_{H^0_{\frak m}(M)}x) + \ell(0:_{H^1_{\frak
m}(M/U_{M}(0))}x)$$ if $d=2$.
\end{lemma}
\begin{proof} Put $\overline{M} = M/U_M(0)$, by Proposition
\ref{M3.2.19} we have the short exact sequence
$$0 \to U_M(0)/xU_M(0) \to U_{M/xM}(0) \to H^0_{\frak m}(\overline{M}/x\overline{M}) \to 0.$$
The case $d \geq 3$, if $\dim U_M(0) < d-1$ then $\dim
U_M(0)/xU_M(0) < d-2$. Thus $\dim U_{M/xM}(0) < d-2$. Hence
$$\delta_{d-2, \dim U_{M/xM}(0)}\mathrm{deg}(I,U_{M/xM}(0))  = 0 =
\delta_{d-1, \dim U_M(0)}\mathrm{deg}(I,U_{M}(0)).$$
If $\dim U_M(0) = d-1$ we have
$\dim U_{M/xM}(0) = d-2 > 0$. By the above short exact sequence we have $\mathrm{deg}(I,U_{M/xM}(0)) =
\mathrm{deg}(I,U_M(0)/xU_M(0))$. By Remark \ref{C3.3.14} (iv) we have $\mathrm{deg}(I,U_{M}(0)) =
\mathrm{deg}(I,U_M(0)/xU_M(0))$. Thus we also have
$$\delta_{d-2, \dim U_{M/xM}(0)}\mathrm{deg}(I,U_{M/xM}(0))  =\delta_{d-1, \dim U_M(0)}\mathrm{deg}(I,U_{M}(0)).$$
The case $d=2$, we have $U_{M/xM}(0)$ has finite length.
Therefore by Remark \ref{C3.3.14} (iv) we have
$$\delta_{0, \dim U_{M/xM}(0)}\mathrm{deg}(I,U_{M/xM}(0)) = \ell(U_{M/xM}(0)) =
\ell(U_M(0)/xU_M(0)) + \ell(H^0_{\frak
m}(\overline{M}/x\overline{M})).$$ If $\dim U_M(0) = 1$, then by Remark
\ref{C3.3.14} (iv) we have
$$\ell(U_M(0)/xU_M(0)) = \mathrm{deg}(I,U_{M}(0))
+ \ell(0:_{U_M(0)}x) = \delta_{1, \dim U_M(0)}\mathrm{deg}(I,U_{M}(0))  +
\ell(0:_{H^0_{\frak m}(M)}x).$$ If $\dim U_M(0) = 0$, then
 we have $U_{M}(0) = H^0_{\frak
m}(M)$ and hence $\delta_{1, \dim U_M(0)}\mathrm{deg}(I,U_{M}(0)) = 0 $. Moreover one can check that $\ell(H^0_{\frak m}(M)/xH^0_{\frak m}(M)) = \ell(0:_{H^0_{\frak m}(M)}x)$. Thus we always have
$$\ell(U_M(0)/xU_M(0)) = \delta_{1, \dim U_M(0)}\mathrm{deg}(I,U_{M}(0)) +
\ell(0:_{H^0_{\frak m}(M)}x).$$ On the other hand, the short exact sequence
$$0 \to \overline{M} \overset{x}{\to} \overline{M} \to \overline{M}/x\overline{M} \to 0$$
induces the exact sequence of local cohomology modules
$$0 \to H^0_{\frak m}(\overline{M}/x\overline{M}) \to H^1_{\frak m}(\overline{M}) \overset{x}{\to}
H^1_{\frak m}(\overline{M}).$$ 
Therefore $\ell(H^0_{\frak
m}(\overline{M}/x\overline{M})) = \ell(0:_{H^1_{\frak
m}(\overline{M})}x)$. Hence
$$\delta_{0, \dim U_{M/xM}(0)}\mathrm{deg}(I,U_{M/xM}(0)) = \delta_{1, \dim U_M(0)}\mathrm{deg}(I,U_{M}(0)) +
\ell(0:_{H^0_{\frak m}(M)}x) + \ell(0:_{H^1_{\frak
m}(M/U_{M}(0))}x).$$
The proof is complete.
\end{proof}
We need one more technical lemma.
\begin{lemma}\label{B3.3.16} Let $M$ be a finitely generated $R$-module of dimension $d \geq 2$. Let $x$ be a parameter element of $M$ such that
 $x \notin \frak p$ for all $\frak p \in
\mathrm{Ass}\, U_M(0) \setminus \{\frak m\}$. Then we can choose a $C$-parameter element $x_d$ of $M$ such that $x$ is a parameter element of $M/x_dM$.
\end{lemma}
\begin{proof}
 If $\dim U_M(0) < d-1$ then $\dim
R/\frak b(M) \leq d-2$ by Remark \ref{C3.1.2} (ii). Therefore we can choose a $C$-parameter element
$x_d$ such that $x$ and $x_d$
is a part of a system of parameters of $M$ by the prime avoidance theorem. Hence $x$ is a parameter element of
$M/x_dM$.\\
We now assume that $\dim U_M(0) = d-1$. Set $\overline{M} = M/U_M(0)$. The short exact sequence
$$0 \to U_M(0) \to M \to \overline{M} \to 0.$$
induces the exact sequence of local cohomology modules
$$ \cdots \to H^i_{\frak m}(U_M(0)) \to H^i_{\frak m}(M) \to H^i_{\frak m}(\overline{M}) \to \cdots.$$
Hence $\frak a_i(M) = \mathrm{Ann}\, H^i_{\frak m}(M) \supseteq
\mathrm{Ann}\, U_M(0) \, \frak a_i(\overline{M})$ for all $i \geq 0$. So
$$\sqrt{\frak b(M)} = \sqrt{\frak a(M)} \supseteq \sqrt{\mathrm{Ann}\, U_M(0)\,\frak a(\overline{M})} =
\sqrt{\mathrm{Ann}\, U_M(0)\,\frak b(\overline{M})}.$$
We claim that $\frak b(M) \nsubseteq \frak q$ for all $\frak q \in
\mathrm{Assh}\, M/xM$. Indeed, we have $\dim R/\frak
b(\overline{M}) \leq d-2$ by Remark \ref{C3.1.2} (ii). Therefore $\frak b(\overline{M}) \nsubseteq \frak q$.
Suppose $\mathrm{Ann}\, U_M(0) \subseteq \frak q$. Then $\frak q \in
\mathrm{Assh}\, U_M(0)$ since $\dim U_M(0) = \dim R/\frak q = d-1$. It contradicts our assumption that $x \notin \frak p$ for all $\frak p \in
\mathrm{Ass}\, U_M(0) \setminus \{\frak m\}$. So
$\mathrm{Ann}\, U_M(0) \nsubseteq \frak q$, and hence $\frak b(M)
\nsubseteq \frak q$ for all $\frak q \in \mathrm{Assh}\, M/xM$. Thus there exists
$x_d \in \frak b(M)^3$ such that $x_d$ is a parameter element of $M/xM$ by the prime avoidance theorem. Such an element $x_d$ satisfies the requirements. The proof is complete.
\end{proof}
We are now ready to prove that the unmixed degrees satisfy the Bertini rule of extended degrees.
\begin{theorem}\label{D3.3.17} Let $M$ be a finitely generated $R$-module of dimension $d$.
Let $x$ be a superficial element of $M$ and of all $U_i(M)$, $1 \leq i
\leq d-1$, with respect to $I$. Then
$$\mathrm{udeg}(I,M) \ge \mathrm{udeg}(I,M/xM).$$
\end{theorem}
\begin{proof} Notice that since $x$ is a superficial element of $U_i(M)$, $1 \leq i \leq d-1$, with respect to $I$ we have $x \notin \frak p$ for all $\frak p \in
\mathrm{Ass}\, U_i(M) \setminus \{\frak m\}$, $1 \leq i \leq
d-1$ by Remark \ref{C3.3.14} (i). The case $d=1$ is clear since $\mathrm{udeg}(I,M) =
\mathrm{deg}(I,M) + \ell(H^0_{\frak m}(M))$ and
$\mathrm{udeg}(I,M/xM) = \ell(M/xM) = \mathrm{deg}(I,M) +
\ell(0:_Mx)$. Suppose $d \geq 2$, by Lemma \ref{B3.3.16} we can choose a part of a $C$-system of parameters
$x_2,\ldots,x_d$ of $M$ such that $x, x_2,\ldots,x_d$
is also a system of parameters of $M$. By Lemma \ref{B3.1.9} we have $x_2, \ldots, x_d$ is a $C$-system of parameters of $M/xM$. Therefore, we have
\begin{eqnarray*}
\mathrm{udeg}(I,M) &=& \mathrm{deg}(I,M) +
\sum_{i=0}^{d-1}\delta_{i, \dim U_i(M)}\mathrm{deg}(I,U_i(M))\\
&=& \mathrm{deg}(I,M) +
\sum_{j=2}^{d+1}\delta_{j-2, \dim U_{M/(x_j,\ldots,x_d)M}(0)}\mathrm{deg}(I,U_{M/(x_j,\ldots,x_d)M}(0)),
\end{eqnarray*}
and
\begin{eqnarray*}
\mathrm{udeg}(I,M/xM) &=& \mathrm{deg}(I,M/xM) +
\sum_{i=0}^{d-2}\delta_{i, \dim U_i(M/xM)}\mathrm{deg}(I,U_i(M/xM))\\
&=& \mathrm{deg}(I,M/xM) +
\sum_{j=3}^{d+1}\delta_{j-3, \dim U_{M/(x,x_j,\ldots,x_d)M}(0)}\mathrm{deg}(I,U_{M/(x,x_j,\ldots,x_d)M}(0)).
\end{eqnarray*}
Since $x$ is a superficial element of $M$ with respect to $I$ we have
$\mathrm{deg}(I,M/xM) = \mathrm{deg}(I,M)$. For $j>3$ we have
$\dim M/(x_j,\ldots,x_d)M = j-1 \geq 3$. By Lemma \ref{B3.3.15} we obtain
$$\delta_{j-2, \dim U_{M/(x_j,\ldots,x_d)M}(0)}\mathrm{deg}(I,U_{M/(x_j,\ldots,x_d)M}(0)) = \delta_{j-3, \dim U_{M/(x,x_j,\ldots,x_d)M}(0)}\mathrm{deg}(I,U_{M/(x,x_j,\ldots,x_d)M}(0))$$
for all $3<j \leq d+1$. For $j=3$, set $M' = M/(x_3,\ldots,x_d)M $ we have $\dim M'
= 2$. By Lemma \ref{B3.3.15} we have
$$\delta_{0, \dim U_{M'/xM'}(0)}\mathrm{deg}(I,U_{M'/xM'}(0)) = \delta_{1, \dim U_{M'}(0)}\mathrm{deg}(I,U_{M'}(0)) +
\ell(0:_{H^0_{\frak m}(M')}x) + \ell(0:_{H^1_{\frak
m}(M'/U_{M'}(0))}x).$$
By Corollary \ref{H3.2.5} we have
$$U_0(M') = H^0_{\frak m}(M'/x_2M') \cong H^0_{\frak m}(M') \oplus H^1_{\frak m}(M'/U_{M'}(0)).$$
So
\begin{eqnarray*}
\delta_{0, \dim U_0(M')}\mathrm{deg}(I,U_0(M')) &=& \ell(H^0_{\frak m}(M')) + \ell(H^1_{\frak m}(M'/U_{M'}(0)))\\
&\ge& \ell(0:_{H^0_{\frak m}(M')}x) + \ell(0:_{H^1_{\frak
m}(M'/U_{M'}(0))}x).
\end{eqnarray*}
Therefore
$$\delta_{0, \dim U_{M'/xM'}(0)}\mathrm{deg}(I,U_{M'/xM'}(0)) \le \delta_{1, \dim U_{M'}(0)}\mathrm{deg}(I,U_{M'}(0)) +\delta_{0, \dim U_0(M')}\mathrm{deg}(I,U_0(M')).$$
More precisely, we have
$$\delta_{0, \dim U_{M/(x,x_3,\ldots,x_d)M}(0)}\mathrm{deg}(I,U_{M/(x,x_3,\ldots,x_d)M}(0)) \le
\sum_{j=2}^{3}\delta_{j-2, \dim U_{M/(x_j,\ldots,x_d)M}(0)}\mathrm{deg}(I,U_{M/(x_j,\ldots,x_d)M}(0)).$$
In conclusion,
$\mathrm{udeg}(I,M) \ge \mathrm{udeg}(I,M/xM)$. The proof is complete.
\end{proof}
\begin{remark}\rm By the prime avoidance theorem we can always choose $x$ satisfying the condition of Theorem \ref{D3.3.17}. Furthermore, according to the above proof we have $\mathrm{udeg}(I,M/xM) = \mathrm{udeg}(I,M)$ provided $x$ annihilates
 $H^0_{\frak m}(M')$ and $H^1_{\frak
m}(M'/U_{M'}(0))$, where $M' = M/(x_3,\ldots,x_d)M$. This is the case if $xU_0(M) = 0$ by Corollary \ref{H3.2.5}.
\end{remark}

By Proposition \ref{M3.3.9}, Theorems \ref{D3.3.8} and \ref{D3.3.17} we have the main result of this section.
\begin{theorem} For every $\frak m$-primary ideal $I$, the unmixed degree $\mathrm{udeg}(I, \bullet)$ is an extended degree on the category of finitely generated $R$-modules $\mathcal{M}(R)$.
\end{theorem}
We next compare the unmixed degree and the homological degree for sequentially Cohen-Macaulay modules.
\begin{remark}\rm
Suppose $(R, \frak m)$ be the homomorphic image of a Gorenstein local ring $S$ of dimension $n$, and $M$ a sequentially Cohen-Macaulay $R$-module. It is easy to see that
$\mathrm{Ext}^i_S(M, S)$ is either a Cohen-Macaulay module or zero module for all $i$. By Theorem
\ref{D3.3.8} we have
$$\mathrm{udeg}(I,M) = \mathrm{adeg}(I,M) = \mathrm{deg}(I,M) + \sum_{i=0}^{d-1} \mathrm{deg}(\mathrm{Ext}^{n-i}_S(M, S))$$
for the last equation see \cite[Theorem 3.11]{NR06}.
Furthermore by \cite[Theorem 3.5]{NR06} we have
$$\mathrm{hdeg}(I,M) = \mathrm{deg}(I,M)  + \sum_{i=0}^{d-1} \binom{d-1}{i}\mathrm{deg}(\mathrm{Ext}^{n-i}_S(M, S)).$$
Therefore $\mathrm{udeg}(I,M) \leq \mathrm{hdeg}(I,M)$. The equality occurs if and only if
$\mathrm{Ext}^{n-i}_S(M, S) = 0$ for all $1 \leq i \leq d-2$. In this case the dimension filtration of $M$ is either $H^0_{\frak m}(M) \subseteq M$ or $H^0_{\frak m}(M) \subseteq U_M(0) \subseteq M$ with dim $U_M(0) = d-1$.
\end{remark}
We close this paper with some examples and an open question.
\begin{example}\rm Let $R = k[[X_1,\ldots,X_4]]/(X_1^2, X_1X_2,
X_1X_3)$ where $k$ is a field and $X_i, 1 \leq i \leq 4,$ are indeterminates. We denote by $x_i$ the image of $X_i$ in $R$. We have that $R$ is a sequentially
Cohen-Macaulay ring of dimension $3$ with the dimension filtration $\mathcal{D}: 0
\subseteq (x_1) \subseteq R$. We have
$$\mathrm{deg}(R) = 1 < \mathrm{adeg}(R) = \mathrm{udeg}(R) = 2 < \mathrm{hdeg}(R) = 3.$$
\end{example}

\begin{example}\rm
Let $R = k[[X_1,\ldots,X_7]]/(X_1,X_2, X_3) \cap (X_4,X_5,X_6)$ where $k$
is a field and $X_i, 1 \leq i \leq 7,$ are indeterminates. It is easy to see that $\mathrm{deg}(R) =  \mathrm{adeg}(R) = 2$. Moreover, we can compute that $\mathrm{hdeg}(R) = 5$ and $\mathrm{udeg}(R) = 4$.
\end{example}
\begin{question}\rm Is it true that $\mathrm{udeg}(I, M) \leq
\mathrm{hdeg}(I, M)$ for all finitely generated $R$-modules $M$ and all $\frak m$-primary ideals $I$?
\end{question}



\end{document}